\documentclass[12pt]{amsart}
\usepackage{amsmath,amssymb}
\usepackage{eufrak}
\usepackage{amsthm}
\usepackage{enumitem}
\usepackage{setspace}
\usepackage{graphicx}
\usepackage{graphics}
\usepackage{color}
\usepackage[dvipsnames]{xcolor}
\usepackage{hyperref}
\usepackage{youngtab}
\usepackage{xcolor}
\usepackage{enumitem}
\usepackage{yfonts}

\hypersetup{ colorlinks,
 linkcolor={blue},  citecolor={BrickRed}, urlcolor={blue}}

\usepackage{amsthm,thmtools,xcolor}

\usepackage{marginnote}
\usepackage{lscape}
\usepackage{faktor}
\usepackage{hyperref}
\usepackage[normalem]{ulem}
\usepackage{bm}
\usepackage{afterpage}
\usepackage{mathrsfs}
\usepackage[all]{xy}
\xyoption{matrix}
\xyoption{arrow}
\usepackage{tikz}
\usetikzlibrary{calc,shadings,patterns}
\usetikzlibrary{matrix}
\usetikzlibrary{arrows}
\usetikzlibrary{matrix}
\usepackage{verbatim}
\usepackage{multirow}
\usetikzlibrary{decorations.pathreplacing}
\usepackage{enumitem,pifont,xcolor}% http://ctan.org/pkg/{enumitem,pifont,xcolor}
\definecolor{myblue}{RGB}{0,29,119}

\newtheorem{theorem}{Theorem}[section]
\newtheorem{proposition}[theorem]{Proposition}
\newtheorem{corollary}[theorem]{Corollary}
\newtheorem{lemma}[theorem]{Lemma}
\theoremstyle{definition}
\newtheorem{definition}[theorem]{Definition}
\newtheorem{example}[theorem]{Example}

\newtheorem{remark}[theorem]{Remark}

\newtheorem*{theorem*}{Theorem}
\makeatletter
\@addtoreset{equation}{section}
\makeatother

\DeclareMathOperator{\End}{End}
\DeclareMathOperator{\coker}{coker}
\DeclareMathOperator{\im}{Im}

\DeclareMathOperator{\gldim}{gl\,\! dim}
\DeclareMathOperator{\domdim}{dom.\! dim}

\DeclareMathOperator{\Hom}{Hom}
\DeclareMathOperator{\ext}{Ext}
\DeclareMathOperator{\Ext}{Ext}
\DeclareMathOperator{\pdim}{p.\!dim}
\DeclareMathOperator{\rank}{Rank}
\DeclareMathOperator{\add}{add}
\DeclareMathOperator{\soc}{soc}
\DeclareMathOperator{\topp}{top}
\DeclareMathOperator{\modd}{mod-\!}
\DeclareMathOperator{\rad}{rad}

\newcommand{\cC}{{\mathcal C}}

\newcommand{\cF}{{\mathcal F}}
\newcommand{\cG}{{\mathcal G}}

\newcommand{\cM}{{\mathcal M}}
\newcommand{\cN}{{\mathcal N}}

\newcommand{\cP}{{\mathcal P}}

\newcommand{\cQ}{{\mathcal Q}}

\newcommand{\cS}{{\mathcal S}}

\renewcommand{\rank}{\operatorname{rank} }

\providecommand{\AMS}{$\mathcal{A}$\kern-.1667em%
\lower.25em\hbox{$\mathcal{M}$}\kern-.125em$\mathcal{S}$}
\begin{document}

%\doublespace
\title[Higher Auslander Algebras arising from Dynkin Quivers]{Higher Auslander Algebras arising from Dynkin Quivers and $n$-Representation Finite Algebras}

\author{Emre SEN}

\address{Iowa City, IA}
\email{\href{emresen641@gmail.com}{emresen641@gmail.com}}

\maketitle

{\let\thefootnote\relax\footnotetext{Keywords: higher cluster-tilting modules; higher Auslander Algebras; Auslander Algebras, Dynkin Quiver, $n$-representation finite, $n$-hereditary, derived category, $n$-cluster tilting subcategory, endomorphism algebras, higher Nakayama algebras
MSC 2020: 16E05, 16E10, 16G20} }

%\begin{itemize}[label={\Pifont{pzd}{\char108}}]
%\item Intor ya eklenecek seyler
%\begin{enumerate}[label=\textcolor{purple}{\arabic*)}]

%\item Gorenstein tekrar yaz, notasyonu degistir.
%\item Kontrol edilmeyen son kisim MadsenMar isleri, 

%\item Changes list: Longer introduction, results are extended i.e. applications to dominant dimension, left-right finitistic and fi dimensions included. Extended bibliography. Submitted. 16G10, 16E10, representation theory
%\end{itemize}

\begin{abstract}
In the derived category of $\modd\mathbb{K}Q$ for a Dynkin quiver $Q$, we construct a full subcategory in a canonical way such that its endomorphism algebra is a higher Auslander algebra of global dimension $3k+2$ for any $k\geq 1$. Furthermore, we extend this construction to higher analogues of representation finite and hereditary algebras. Specifically, if $M$ is an $n$-cluster tilting object in the bounded derived category of an $n$-representation finite and $n$-hereditary algebra, then we construct a full subcategory in a canonical way such that its endomorphism algebra is a higher Auslander algebra of global dimension $(n+2)k+n+1$ for any $k\geq 1$.

%In the derived category of $\modd\mathbb{K}Q$ for a Dynkin quiver $Q$, we construct a full subcategory in a canonical way, so that its endomorphism algebra is a higher Auslander algebra of global dimension $3k+2$ for any $k\geq 1$. Furthermore, we extend this construction for higher analogues of representation finite and hereditary algebras. Specifically, if $M$ is an n-cluster tilting object in the bounded derived category of n-representation finite and n-hereditary algebra, then we construct a full subcategory in a canonical way, so that its endomorphism algebra is a higher Auslander algebra of global dimension $(n+2)k+n+1$ for any $k\geq 1$.
As an application, we revisit the higher Auslander correspondence. First, we describe the corresponding module categories that admit higher cluster-tilting objects, and then we discuss their relationship with certain full subcategories of the derived category. Consequently, we obtain a vast family of $n$-representation finite and $n$-hereditary algebras whose $n$-cluster tilting objects are always minimal generator–cogenerators. Moreover, the resulting algebras can be realized as endomorphism algebras of certain full subcategories of (higher) cluster categories.
  
 \end{abstract}

\tableofcontents
\section{Introduction}

O.~Iyama introduced the higher Auslander correspondence in \cite{iyama2007auslandercorrespondence}, showing that a finite-dimensional algebra $A$ over an algebraically closed field $\mathbb{K}$ satisfying
\begin{align}
\gldim A \leq d+1 \leq \domdim A
\end{align}
for some $d \geq 1$ can be realized bijectively as $A := \End_{B}(\cM)$, where $\cM$ is a $d$-cluster-tilting object in the category of finitely generated $B$-modules for some algebra $B$. Since then, the classification of $d$-cluster-tilting modules for a given class of algebras, or the characterization of higher Auslander algebras within a class of algebras, has been a challenging problem, even for well-understood module categories. In the works \cite{sen2020nakayama}, \cite{senDefectlinear}, and \cite{ringel2022linear}, Nakayama algebras that are higher Auslander have been studied. In \cite{vaso2019n} and \cite{darpo2020d}, Nakayama algebras admitting higher cluster-tilting objects were investigated.

%In this work, we follow another approach; instead of $d$-cluster-tilting subcategories, we use certain full subcategories of the bounded derived categories to construct higher Auslander algebras. The main subject of this work is the following category.

In this work, we follow a different approach: instead of considering $d$-cluster-tilting subcategories, we utilize specific full subcategories of the bounded derived category to construct higher Auslander algebras. The main focus of this work is the following category.

\begin{definition}\label{definition dynkin category}
Let $\modd\Lambda$ be the category of finitely generated left modules over the path algebra $\Lambda=\mathbb{K}Q$ of a Dynkin quiver $Q$. Consider the full subcategory of the bounded derived category $D^b(\modd\Lambda)$, denoted by $\cS^{k}$, whose objects are
\begin{align*}
\left\{M \,\middle|\, M = X[j],\, 0 \leq j \leq k,\, \forall X \in \modd\Lambda \right\},
\end{align*}
and whose morphisms are given by
\begin{align*}
\Hom_{\cS^k}(X[i],Y[j]) =
\begin{cases}
\Hom_{\Lambda}(X,Y), & \quad \text{if } i = j,\\
\Ext^1_{\Lambda}(X,Y), & \quad \text{if } i = j - 1,\\
0, & \quad \text{otherwise.}
\end{cases}
\end{align*}
\end{definition}

Consider the object
\begin{align*}
\cP = \bigoplus\limits_{\substack{X \in \text{Ind}\Lambda \\ 0 \leq j \leq k}} X[j].
\end{align*}

\noindent
Here, $\text{Ind}\text{-}\Lambda$ denotes a set of representatives of the isomorphism classes of indecomposable $\Lambda$-modules.
It follows that $\add\cP = \cS^k$.
The first result we present is the following:

\begin{theorem}\label{THM dynkin}
The algebra $\Gamma^k := \End_{\cS^k}(\cP)$ is a higher Auslander algebra of global dimension $3k+2$ for $k \geq 1$. Moreover, the opposite quiver of $\Gamma^k$ coincides with the Auslander--Reiten quiver of $\cS^k$.
\end{theorem}
First, recall that for any algebra of finite representation type, the endomorphism algebra of an additive generator of the module category has global dimension at most two and dominant dimension at least two, by the classical Auslander correspondence. The theorem stated above mimics this construction by taking the additive generator of the full subcategory $\mathcal{S}^k$ of the derived category and thereby producing higher Auslander algebras. Consequently, the quiver of $(\Gamma^k)^{op}$ coincides with the Auslander--Reiten quiver of the category $\mathcal{S}^k$. This implies that the quiver of $\Gamma^k$ is obtained by appropriately gluing $k+1$ copies of the Auslander--Reiten quiver of $\modd\Lambda^{op}$. We prove this in Section~\ref{section dynkin}.

Before stating our next result, we recall that a subcategory $\cC$ of $D^b(\modd\Lambda)$ is called \emph{convex} if, for every $X, Y \in \cC$, whenever there exists a nonzero composition $X \rightarrow Z \rightarrow Y$, then $Z \in \cC$. Algebras $\Gamma^k$ are unique in the following sense.

\begin{theorem}\label{thm uniqueness gamma dynkin}
Let $\cC = \add \cG$ be a convex subcategory of $D^b(\modd\Lambda)$ such that
\[
\cS^0 \varsubsetneq \cC \varsubsetneq \cS^k.
\]
The following statements are equivalent:
\begin{enumerate}[label=\roman*)]
\item $\End_{D^b(\modd\Lambda)}(\cG)$ is a higher Auslander algebra;
\item $\End_{D^b(\modd\Lambda)}(\cG) \cong \Gamma^i$ for some $i$, $1 \le i \le k-1$;
\item $\cC \cong \cS^i$ for some $i$, $1 \le i \le k-1$.
\end{enumerate}
\end{theorem}

Furthermore, in Section~\ref{section uniqueness}, we provide upper and lower bounds for the dominant and global dimensions of $\End_{D^b(\modd\Lambda)}(\cG)$ when it is not a higher Auslander algebra.\\

We observe that the construction in Definition~\ref{definition dynkin category} is not restricted to Dynkin quivers. Recall that a finite dimensional algebra $A$ is called $n$-representation finite if it possesses a unique $n$-cluster tilting object $\mathcal{M}$ such that
\begin{align}
\mathcal{M} := \add\left(\bigoplus_{j \geq 0} \tau_n^j(\text{D}A)\right)
\end{align}
where $\tau_n := \tau\Omega^{n-1}$ denotes the higher Auslander--Reiten translate and $\text{D} = \Hom_{\mathbb{K}}(-,\mathbb{K})$ is the $\mathbb{K}$ dual. Furthermore, $A$ is called an $n$-hereditary algebra if the global dimension of $A$ is $n$. In this case, the subcategory
\begin{align}
\mathcal{M}[n\mathbb{Z}] := \add\left(X[\ell n]\,\vert\, X\in\mathcal{M},\, \ell\in\mathbb{Z}\right)
\end{align}
is an $n$-cluster tilting subcategory of $D^b(\modd A)$. Moreover, it forms an $(n+2)$-angulated category \cite{geiss2013n, iyama2011clusterMAINpaper}. The higher analogue of Theorem~\ref{THM dynkin} is as follows.

\begin{theorem}\label{THM higher hereditary}
Let $A$ be an $n$-representation finite and $n$-hereditary algebra where $\cM$ is the unique $n$-cluster tilting object. Consider the full subcategory of the bounded derived category $D^b(\modd A)$ denoted by $\cM^k$ whose objects are
\begin{align*}
\left\{M\,\vert\, M = X[jn],\, 0 \leq j \leq k,\, \forall X \in \cM\right\},
\end{align*}
and whose morphisms are given by
\begin{align*}
\Hom_{\cM^k}(X[in],Y[jn])=
\begin{cases}
\Hom_A(X,Y), & \text{if } i=j,\\[4pt]
\Ext_A^n(X,Y), & \text{if } i=j-1,\\[4pt]
0, & \text{otherwise}.
\end{cases}
\end{align*}
Then, the algebra $\Gamma^k := \End_{\cM^k}(\cP)$ is a higher Auslander algebra of global dimension $(n+2)k+n+1$, where $\add\cP = \cM^k$. Moreover, the opposite quiver of $\Gamma^k$ coincides with the Auslander--Reiten quiver of $\cM^k$ in $\cM[n\mathbb{Z}] \subset D^b(\modd A)$.
\end{theorem}

In other words, we can glue higher Auslander algebras arising from $n$-representation finite $n$-hereditary algebras in a suitable way such that the resulting algebra is again higher Auslander. We give the proof in Section~\ref{section n-rep finite}. Similar to Theorem~\ref{thm uniqueness gamma dynkin}, algebras $\Gamma^k$ are unique in the following sense:

\begin{theorem}\label{thm uniqueness gamma higher}
Let $\cC = \add \cG$ be a convex subcategory of $\cM[n\mathbb{Z}]$ such that $\cM^0 \varsubsetneq \cC \varsubsetneq \cM^k$.
The following statements are equivalent:
\begin{enumerate}[label=\roman*)]
\item $\End_{D^b(\modd A)}(\cG)$ is a higher Auslander algebra;
\item $\End_{D^b(\modd A)}(\cG) \cong \Gamma^i$ for some $i$, $1 \le i \le k-1$;
\item $\cC \cong \cM^i$ for some $i$, $1 \le i \le k-1$.
\end{enumerate}
\end{theorem}

In general, the construction or classification of $n$-representation finite or $n$-hereditary algebras remain open problems. Significant progress has been made in this area, for example \cite{iyama2011clusterMAINpaper}, \cite{herschend2011fractionally}, \cite{herschend2014iyamaopperman}, \cite{iyama2008auslanderreitenrevisited}, \cite{vaso2019n}, and \cite{haugland2022role}. We remark that the algebras discussed in Theorems~\ref{THM dynkin} and~\ref{THM higher hereditary} arise from $n$-hereditary $n$-representation finite algebras.

\begin{theorem}\label{THM sigma k algebra}
1) Let $\cQ$ be the projective--injective object of $\Gamma^k = \End_{\cS^k}(\cP)$ for some $k \geq 1$. Then $\Sigma^k := \End_{\Gamma^k}(\cQ)$ is a $d$-representation finite and $d$-hereditary algebra, where $d = 3k + 1$. Any $d$-cluster tilting subcategory of $\modd \Sigma^k$ is of the form $\add\left(D\Sigma^k \oplus \tau_{d} D\Sigma^k\right)$,
which is the minimal generator-cogenerator of $\modd \Sigma^k$. Moreover, $\Sigma^k$ can be realized as the endomorphism algebra of the fundamental domain of the $k$-cluster category $D^b(\modd \Lambda)/\tau^{-1}[k]$.

2) Let $\cQ$ be the projective--injective object of $\Gamma^k = \End_{\cM^k}(\cP)$ for some $k \geq 1$. Then $\Sigma^k := \End_{\Gamma^k}(\cQ)$ is a $d$-representation finite and $d$-hereditary algebra, where $d = (n+2)k + n$. Any $d$-cluster tilting subcategory of $\modd \Sigma^k$ is of the form $\add\left(D\Sigma^k \oplus \tau_{d} D\Sigma^k\right)$,
which is the minimal generator--cogenerator of $\modd \Sigma^k$. Moreover, $\Sigma^1$ can be realized as the endomorphism algebra of the fundamental domain of the higher cluster category $\cM \oplus A[n]$.
\end{theorem}

Cluster categories were introduced in \cite{buan2006tilting} to categorify cluster algebras. They are defined as the orbit category $D^b(\modd\Lambda)/\tau^{-1}[1]$. Similarly, $m$-cluster categories were defined in \cite{ASSEM2006548}, \cite{ASSEM2008884} as the orbit categories $D^b(\modd\Lambda)/\tau^{-1}[m]$. Thus, the fundamental domain of the $m$-cluster category is 
\begin{align*}
\Lambda[m] \oplus \bigoplus\limits_{\substack{X \in \text{Ind}\Lambda \\ 0 \le j \le m-1}} X[j]
\end{align*}
where $\Lambda = \mathbb{K}Q$ and $Q$ is Dynkin. For a detailed exposition, see \cite{reiten2010cluster}.

Cluster categories for non-hereditary algebras were introduced by C. Amiot in \cite{Ami09}. Within Amiot's cluster categories, higher cluster categories were introduced in \cite{oppermann2012higher}, whose fundamental domains are of the form $\cM \oplus A[n]$. For our purposes, we consider
\begin{align*}
A[mn] \oplus \bigoplus\limits_{\substack{X \in \cM \\ 0 \le j \le m-1}} X[jn]
\end{align*}
as fundamental domains of higher cluster categories, where the corresponding endofunctor is $\tau^{-n}[mn]:\cM[n\mathbb{Z}]\rightarrow\cM[n\mathbb{Z}]$.

As a consequence of Theorem~\ref{THM sigma k algebra}, we obtain another connection between cluster theory and higher dimensional homological algebra. In particular, it provides a rich source of $d$-representation finite and $d$-hereditary algebras. We describe the quiver of $\End_{\Gamma^k}(\cQ)$ in Section~\ref{section endomorphism algebras of gamma}. 

As an application of Theorems~\ref{thm uniqueness gamma dynkin} and~\ref{thm uniqueness gamma higher}, we show that the algebras $\Sigma^k$ are unique in the following sense.

\begin{theorem}\label{thm uniqueness sigma all} 1) Let $\cC=\add\cG$ be a convex subcategory of $D^b(\modd\Lambda)$ such that 
\begin{align*} \cS^0\varsubsetneq \cC \varsubsetneq \cS^k.
 \end{align*} Then, $\End_{D^b(\modd\Lambda)}(\cG)$ is a $d$-representation finite algebra if and only if $\cC\cong\cS^i\oplus\Lambda[i+1] $ for some $i$ $0\leq i\leq k-1$.\\ 2) Let $\cC=\add\cG$ be a convex subcategory of $\cM[n\mathbb{Z}]$ such that \begin{align*} \cM^0\varsubsetneq \cC \varsubsetneq \cM^k.
  \end{align*} Then, $\End_{D^b(\modd A)}(\cC)$ is a $d$-representation finite algebra if and only if  $\cC\cong\cM^i\oplus A[(i+1)n]$ for some $i$ $0\leq i\leq k-1$. \end{theorem}

It is natural to ask what the other cluster-tilting objects in $\modd \Sigma^k$ are. We are able to describe them in a special case: higher Nakayama algebras, discussed in Section \ref{applications to nakayama}. Specifically, we present a class of higher Nakayama algebras that are $d$-representation finite and isomorphic to certain $\Sigma^k$, for which we describe the Kupisch series in Proposition \ref{prop B is s-rep finite numerical}. Consequently, they contain $d\mathbb{Z}$-cluster tilting objects.

In the following section, we give preliminaries. Then, in sections \ref{section dynkin},\ref{section n-rep finite},\ref{section endomorphism algebras of gamma} and \ref{section uniqueness} we prove our main results. The last section is devoted for final remarks and examples.

\subsection{Acknowledgments} We are grateful to K. Igusa, O. Iyama, P. J{\o}rgensen, B. Keller, and G. Todorov for their interest in this work. We appreciate the help and support of O. Iyama, P. J{\o}rgensen, and G. Todorov in various discussions on the material, which are leading to other works.

\section{Preliminaries}\label{section prelim}
\subsection{Derived Category for Dynkin Case}
Let $\modd\Lambda$ be the module category of the algebra $\Lambda=\mathbb{K}Q$ for a Dynkin quiver Q. The category $\modd\Lambda$ and the bounded derived category $D^b(\modd\Lambda)$ are well understood. We refer to \cite{happel1988triangulated} for details. For the objects of $\modd\Lambda$, we do not type the shift functor $[0]$. Consider an exact sequence in $\modd\Lambda$
 \begin{align*}
 0\rightarrow A\rightarrow B\rightarrow C\rightarrow 0.
 \end{align*}
 This can be completed into the triangle 
  \begin{align*}
 A\rightarrow B\rightarrow C\rightarrow A[1]
 \end{align*}
 in $D^b(\modd\Lambda)$, since it is a triangulated category. So, any exact sequence gives rise to the sequence 
\begin{align*}
\cdots\rightarrow C[-1]\rightarrow A\rightarrow B\rightarrow C\rightarrow A[1]\rightarrow B[1]\rightarrow C[1]\rightarrow A[2]\rightarrow \cdots .
\end{align*} 
The category $\cS^k$ contains the sequence which we obtain by dropping negative shifts and shifts greater than $k$, so we get 
\begin{align*}
0\rightarrow A\rightarrow B\rightarrow C\rightarrow A[1]\rightarrow\cdots\rightarrow A[k]\rightarrow B[k]\rightarrow C[k]\rightarrow 0.
\end{align*}
We recall that the category $\cS^k$ is the full subcategory of $D^b(\modd\Lambda)$. We fix the notation, $G$ is additive generator of $\modd\Lambda$, i.e., $\add G=\modd\Lambda$, $\cP$ is additive generator of $\cS^k$, i.e. $\add\cP=\cS^k$.

\begin{lemma}\label{lemma hom functor is exact on triangles} Let $0\rightarrow M\rightarrow K\rightarrow N\rightarrow 0$ be an exact sequence in $\modd\Lambda$. Then the functor $\Hom_{\cS^k}(\cP,-)$ is left exact on the sequence
\begin{align} \label{eq in lemma prelim}
0\rightarrow M\rightarrow K\rightarrow N\rightarrow M[1]\rightarrow \cdots \rightarrow K[k]\rightarrow N[k]\rightarrow 0.
\end{align}
\end{lemma}

\begin{proof}
The functor $\Hom_{\Lambda}(G,-)$ induces the long exact sequence
\begin{align*}
0\rightarrow \Hom_{\Lambda}(G, M)\rightarrow\Hom_{\Lambda}(G, K)\rightarrow\Hom_{\Lambda}(G, N)\rightarrow\ext^1_{\Lambda}(G, M)\rightarrow \ext^1_{\Lambda}(G,K)\rightarrow \ext^1_{\Lambda}(G,N)\rightarrow 0
\end{align*}
since $\gldim\Lambda=1$. So we can construct split exact sequence of the form
\begin{gather*}
0\rightarrow \Hom_{\Lambda}(G, M)\rightarrow\Hom_{\Lambda}(G, K)\rightarrow\Hom_{\Lambda}(G, N)\rightarrow\ext^1_{\Lambda}(G, M)\oplus\Hom_{\Lambda}(G,M)\rightarrow\\ \ext^1_{\Lambda}(G,K)\oplus\Hom_{\Lambda}(G,K)\rightarrow \ext^1_{\Lambda}(G,N)\oplus\Hom_{\Lambda}(G,N)\rightarrow 
\ext^1_{\Lambda}(G, M)\oplus\Hom_{\Lambda}(G,M)\rightarrow\\ \ext^1_{\Lambda}(G,K)\oplus\Hom_{\Lambda}(G,K)\rightarrow \ext^1_{\Lambda}(G,N)\oplus\Hom_{\Lambda}(G,N)\rightarrow\ext^1_{\Lambda}(G, M)\oplus\Hom_{\Lambda}(G,M)\rightarrow\\\vdots\\ \ext^1_{\Lambda}(G, M)\oplus\Hom_{\Lambda}(G,M)\rightarrow \ext^1_{\Lambda}(G,K)\oplus\Hom_{\Lambda}(G,K)\rightarrow \ext^1_{\Lambda}(G,N)\oplus\Hom_{\Lambda}(G,N)\rightarrow\\ \ext^1_{\Lambda}(G, M)\rightarrow \ext^1_{\Lambda}(G,K)\rightarrow \ext^1_{\Lambda}(G,N)\rightarrow 0
\end{gather*}
Since $\Hom_{\cS^k}(\cP,X[j])\cong\Hom_{\cS^k}(\bigoplus _{0\leq i\leq k}G[i],X[j])\cong \Hom_{\Lambda}(G,X)\oplus \ext^1_{\Lambda}(G,X)$ for any $X[j]$ $1\leq j\leq k$ where $X\in\modd\Lambda$, $\Hom_{\cS^k}(\cP,-)$ applied to \ref{eq in lemma prelim} is isomorphic to the sequence above, hence it is left exact.
\end{proof}

\begin{corollary} Let $0\rightarrow M\rightarrow I_0\rightarrow I_1\rightarrow 0$ be the injective coresolution of $M\in\modd\Lambda$. Then the functor $\Hom_{\cS^k}(\cP,-)$ is left exact on the sequence 
\begin{align*}
0\rightarrow M\rightarrow I_0\rightarrow I_1\rightarrow M[1]\rightarrow I_0[1]\rightarrow\cdots I_1[k]\rightarrow 0.
\end{align*}
\end{corollary}
\begin{proof}
For any $M\in\modd\Lambda$, not injective, the injective copresentation is exact since $\Lambda$ is hereditary. By lemma \ref{lemma hom functor is exact on triangles}, the result follows.

\iffalse
That it 
is also minimal follows easily from the facts that 0 —> A —> Iq —* h 
is a minimal injective Л-copresentation and that Нотд(М, ):modA —* 
&(Гм) is an equivalence of categories. 

\fi

\end{proof}

\begin{remark}\label{remark serre functor on triangulated cat}
We recall Nakayama functor $\nu$ and its inverse $\nu^{-1}$
\begin{gather*}
\nu:=D\Hom_{\Lambda}(-,\Lambda): \modd\Lambda\rightarrow\modd\Lambda\\
\nu^{-1}:=\Hom_{\Lambda^{op}}(D-,\Lambda):\modd\Lambda\rightarrow \modd\Lambda
\end{gather*}
Derived Nakayama functor is
\begin{gather*} 
\nu:=DR\Hom_{\Lambda}(-,\Lambda): D^b(\modd\Lambda)\rightarrow D^b(\modd\Lambda)\\
\nu^{-1}:=R\Hom_{\Lambda^{op}}(D-,\Lambda):D^b(\modd\Lambda)\rightarrow D^b(\modd\Lambda)
\end{gather*}
which gives Serre functor of $D^b(\modd\Lambda)$, i.e.,
there exists a functorial isomorphism
\begin{align*}
\Hom_{D^b(\modd\Lambda)}(X,Y)\cong D\Hom_{D^b(\modd\Lambda)}(Y,\nu(X)).
\end{align*}

Since $D^b(\modd\Lambda)$ admits Auslander-Reiten triangles, by formula \cite[Prop 4.10]{happel1988triangulated} we get an autoequivalence, the Auslander-Reiten translation $\tau$, given by 
\begin{align*}
\Hom_{D^b(\modd\Lambda)}(-,(\tau X)[1])\cong D\Hom_{D^b(\modd\Lambda)}(X,-)
\end{align*}
\cite{keller2005triangulated}, \cite{happel1988triangulated}, where $\tau$ is Auslander-Reiten translate in $\modd\Lambda$.
\end{remark}
\subsection{Derived Category for $n$-Representation finite case} 
Let $A$ be a finite dimensional algebra. Following \cite{iyama2007auslandercorrespondence}, let $\cM$ be a subcategory of $\modd A$. $\cM$ is called $n$-rigid if $\ext^i_{A}(\cM,\cM)=0$ for any $0<i<n$. $\cM$ is called $n$-cluster tilting subcategory if it is functorially finite and
\begin{align*}
\cM&=\{X\in\modd A\vert \ext^i_{A}(X,\cM)=0, 0<i<n\}\\
&=\{X\in\modd A\vert \ext^i_{A}(\cM,X)=0, 0<i<n\}
\end{align*}

Similarly, $n$-cluster tilting subcategories of derived categories introduced in \cite{iyama2011clusterMAINpaper} which we recall. $\cN$ of $D^b(\modd A)$ is $n$-cluster tilting subcategory if
\begin{align*}
\cN&=\{X\in D^b(\modd A) \vert \Hom_{D}(X,\cN[i])=0, 0<i<n\}\\
&=\{X\in D^b(\modd A) \vert \Hom_{D}(\cN[i],X)=0, 0<i<n\}.
\end{align*}

O. Iyama introduced "$n$-complete algebras" in \cite{iyama2011clusterMAINpaper} which is called now $n$-representation finite algebras. Later, in \cite{herschend2014iyamaopperman} $n$-representation infinite algebras were introduced and then both class of algebras are called $n$-hereditary algebras. Hence, to avoid any terminological complications, we say that an algebra $A$ is $n$-representation finite and $n$-hereditary if it admits a unique $n$-cluster tilting subcategory $\cM$ which is always of the form $\cM=\add(\bigoplus_j\tau^j_n(DA))$ together with the assumption $\gldim A\leq n$ where $\tau_n:=\tau\Omega^{n-1}$ is $n$-Auslander-Reiten translate, $\Omega:\modd A\rightarrow \modd A$ is syzygy functor.

In our set up, $\cM[n\mathbb{Z}]$ is $n$-cluster tilting subcategory of $D^b(\modd A)$. This category is $(n+2)$-angulated category (\cite[Theorem 1]{geiss2013n}), hence any long exact sequence 
\begin{align*}
0\rightarrow X\rightarrow A_1\rightarrow A_2\rightarrow\cdots A_n\rightarrow Y\rightarrow 0
\end{align*}
in $\cM$ can be completed into $(n+2)$ angle
\begin{align*}
 X\rightarrow A_1\rightarrow A_2\rightarrow\cdots A_n\rightarrow Y\rightarrow X[n] 
\end{align*}
in $\cN$, therefore we get the sequence of the form 
\begin{align*}
 Y[-n]\rightarrow X\rightarrow A_1\rightarrow A_2\rightarrow\cdots A_n\rightarrow Y\rightarrow X[n] \rightarrow\cdots\rightarrow Y[n]\rightarrow X[2n]\rightarrow \cdots.
\end{align*}
in $\cM[n\mathbb{Z}]$.

\begin{definition}\label{def the higher category M_k}
Let $A$ be an $n$-representation finite and $n$-hereditary algebra where $\cM$ is the unique $n$-cluster tilting object. Consider the full subcategory of the bounded derived category $D^b(\modd A)$ denoted by $\cM^k$ whose objects are
\begin{align*}
\left\{M\,\vert\, M=X[jn], 0\leq j\leq k, \forall X\in\cM\right\},
\end{align*}
and whose morphisms are given by
\begin{align*}
\Hom_{\cM^k}(X[in],Y[jn])=\begin{cases}
\Hom_{A}(X,Y),\,\quad \text{if} \,\quad i=j \\
\ext^n_{A}(X,Y),\,\,\,\quad \text{if}\,\quad i=j-1\\
\quad 0\,\quad\quad\quad\quad \text{otherwise}.
\end{cases}
\end{align*}
\end{definition}

The category $\cM^k$ contains the following sequence which is obtained by dropping negative shifts and shifts greater than $kn$, so we get 
\begin{align*}
0\rightarrow X\rightarrow A_1\rightarrow\cdots A_n\rightarrow Y\rightarrow X[n] \rightarrow\cdots\rightarrow Y[kn]\rightarrow 0.
\end{align*}
We use the same notation $\cP$ where $\add\cP=\cM^k$ and $G\in\modd A$ is additive generator of $\cM$, i.e., $\add G=\cM$.

We recall \cite[Lemma 3.5]{iyama2011clusterMAINpaper}.

\begin{lemma}\label{lemma iyama's ext n}
Let $A$ be a finite dimensional algebra such that $\gldim A\leq n$. Let $X\in\modd A$ and 
\begin{align*}
\xymatrix{
0\ar[r]& X_0\ar[r]& X_1\ar[r]&\cdots\ar[r]& X_n\ar[r]& X_{n+1}\ar[r]&0}
\end{align*}
an exact sequence in $\modd A$ with $X_i\in\add X$.
If $W\in\modd A$ satisfies $\ext^i_{A}(W,X)=0$ for any $0<i<n$, then there is an exact sequence
\begin{align*}
0\rightarrow\Hom_{A}(W,X_0)\rightarrow\Hom_{A}(W, X_1)\cdots\rightarrow \Hom_{A}(W,X_{n+1})\rightarrow\\
\ext^n(W,X_0)\rightarrow\ext^n(W,X_1)\rightarrow\cdots\rightarrow\ext^n(W,X_{n+1})\rightarrow 0
\end{align*}
\end{lemma}

\begin{lemma}\label{lemma hom functor is exact on n angles higher version} Let A be $n$-representation finite $n$-hereditary algebra with $n$-cluster tilting subcategory $\cM$. Let
$\xymatrix{
0\ar[r]& X_0\ar[r]& X_1\ar[r]&\cdots\ar[r]& X_n\ar[r]& X_{n+1}\ar[r]&0}$\\
be an exact sequence in $\cM$. Then $\Hom_{\cM^k}(\cP,-)$ is left exact on the sequence

\begin{align}\label{eq in lemma hom functor is exact in n angles}
0\rightarrow X_0\rightarrow \cdots X_{n+1}\rightarrow X_{0}[n]\rightarrow \cdots \rightarrow X_{n+1}[kn]\rightarrow 0.
\end{align}
\end{lemma}
\begin{proof}
Similar to the proof of lemma \ref{lemma hom functor is exact on triangles}: by definition \ref{def the higher category M_k}, $\Hom_{\cM^k}(\cP,X[jn])=\Hom_{A}(G,X)\oplus \ext^n_{A}(G,X)$ for any $X\in\cM$ because $\cP=\bigoplus_{0\leq j\leq k} G[jn]$. If we apply the functor $\Hom_{\cM^k}(\cP,-)$ to \ref{eq in lemma hom functor is exact in n angles}, the resulting sequence is isomorphic to
\begin{gather*}
0\rightarrow \Hom_{A}(G, X_0)\rightarrow \Hom_{A}(G,X_1)\rightarrow\cdots\rightarrow\Hom_{A}(G, X_{n+1})\rightarrow \\
 \ext^n_{A}(G,X_0)\oplus \Hom_{A}(G, X_0)\rightarrow \ext^n_{A}(G,X_1)\oplus \Hom_{A}(G, X_1)\rightarrow\cdots
 \\  \ext^n_{A}(G,X_{n+1})\oplus \Hom_{A}(G, X_{n+1})\rightarrow \ext^n_{A}(G,X_{0})\oplus \Hom_{A}(G, X_{0})\rightarrow\\
\vdots\\ 
\ext^n_{A}(G,X_0)\rightarrow \ext^n_{A}(G,X_1)\rightarrow\cdots\rightarrow \ext^n_{A}(G,X_{n+1})
\end{gather*}
which is exact by lemma \ref{lemma iyama's ext n}.
\end{proof}

\begin{corollary}\label{cor to lemma hom is exact on triangles} Let $0\rightarrow M\rightarrow I_0\rightarrow I_1\rightarrow \cdots I_n\rightarrow 0$ be the injective coresolution of $M\in\cM$. Then the functor $\Hom_{\cM^k}(\cP,-)$ is left exact on the sequence 
\begin{align*}
0\rightarrow M\rightarrow I_0\rightarrow I_1\rightarrow \cdots\rightarrow I_n\rightarrow M[1]\rightarrow I_0[1]\rightarrow \cdots I_{n-1}[kn]\rightarrow I_n[kn].
\end{align*}
\end{corollary}
\begin{proof}
Since the injective coresolution of $M\in\cM$ is exact, by lemma \ref{lemma hom functor is exact on n angles higher version} claim holds.
\end{proof}
\iffalse
\begin{proof} 
The sequence $0\rightarrow M\rightarrow I_0\rightarrow I_1\rightarrow \cdots\rightarrow I_n$ induces the long exact sequence 
\begin{align}
0\rightarrow \Hom_{A}(\cM, M)\rightarrow \Hom_{A}(\cM,I_0)\rightarrow\Hom_{A}(\cM, I_1)\rightarrow \ext^n_{A}(\cM,M)\rightarrow 0
\end{align}
by .
Therefore the following sequence is exact
\begin{gather*}
0\rightarrow \Hom_{A}(\cM, M)\rightarrow \Hom_{A}(\cM,I_0)\rightarrow\Hom_{A}(\cM, I_1)\rightarrow \ext^n_{A}(\cM,M)\oplus \Hom_{A}(\cM, M)\rightarrow\\
\Hom_{A}(\cM, M)\rightarrow \Hom_{A}(\cM,I_0)\rightarrow\Hom_{A}(\cM, I_1)\rightarrow \ext^n_{A}(\cM,M)\oplus \Hom_{A}(\cM, M)\rightarrow\\
\vdots\\
\Hom_{A}(\cM, M)\rightarrow \Hom_{A}(\cM,I_0)\rightarrow\Hom_{A}(\cM, I_1)\rightarrow \ext^n_{A}(\cM,M)\oplus \Hom_{A}(\cM, M)\rightarrow \\
 \Hom_{A}(\cM, M)\rightarrow \Hom_{A}(\cM,I_0)\rightarrow\Hom_{A}(\cM, I_1)\rightarrow \ext^n_{A}(\cM,M)\rightarrow 0.
\end{gather*}

Similar to lemma , we have the following arguments. By lemma  $\Hom_{\cM^k}(\cP,I[jn])$ for any $0\leq j\leq k$ is isomorphic to $\Hom_{A}(\cM,I)$ for any injective $I$. Moreover, we have

\begin{align}
\Hom_{\cM^k}(\cP,M[jn])&\cong\Hom_{\cM^k}(\cM[(j-1)n]\oplus \cM[jn],M[jn])\\
&\cong \Hom_{A}(\cM,M)\oplus \ext^n_{\Lambda}(\cM,M)
\end{align}
By , we get the desired result.
\end{proof}

\fi
\begin{remark}\label{remark serre functor on higher angulated cat}
Let $\nu$ be the Nakayama functor (remark \ref{remark serre functor on triangulated cat}). We recall that $\nu_n:=\nu\circ[-n]:D^b(\modd A)\rightarrow D^b(\modd A)$ gives autoequivalence of $D^b(\modd A)$ and satisfies
\begin{enumerate}[label=\roman*)]
\item For any $i\in\mathbb{Z}$, there is a functorial isomorphism
\begin{align*}
\Hom_{D^b(\modd A)}(X,Y[i])\cong D\Hom_{D^b(\modd A)}(Y,\nu_n(X)[n-i])
\end{align*}
\item The diagram
\begin{align*}
\xymatrix{D^b(\modd A)\ar[d]^{D}\ar[r]^{\nu_n}&D^b(\modd A)\ar[d]^D\\
D^b(\modd A^{op})\ar[r]^{\nu^{-1}_n}&D^b(\modd A^{op})}
\end{align*}
commutes \cite[Obs. 2.1]{herschend2014iyamaopperman}
\end{enumerate}
As in remark \ref{remark serre functor on triangulated cat}, $\cM[n\mathbb{Z}]\subset D^b(\modd A)$ admits an autoequivalence, the $n$-Auslander-Reiten translation $\tau_n$, given by 
\begin{align*}
\Hom_{D^b(\modd A)}(-,(\tau_n X)[n])\cong D\Hom_{D^b(\modd A)}(X,-)
\end{align*}
\end{remark}

\subsection{Auslander \& Higher Auslander Algebras}
Let $\Lambda$ be a finite dimensional artin algebra  algebra. Then, the Auslander correspondence states that any algebra $B$ satisfying
\begin{align*}
\gldim B\leq 2\leq \domdim B
\end{align*}
where $\gldim$ and $\domdim$
 stands for global dimension and dominant dimension respectively, can be obtained as $B=\End_{\Lambda}(G)$ where $\Lambda$ is of finite representation type, $\add G=\modd\Lambda$. We recall that the dominant dimension of $B$-module $A$ is the maximum integer (or $\infty$) having the property that if $0\rightarrow A\rightarrow I_0\rightarrow I_1\rightarrow\cdots I_t\rightarrow\cdots$ is the minimal injective coresolution of $A$, then $I_j$ is projective for all $j<t$ (or $\infty$). 
 
  O. Iyama introduced higher Auslander correspondence which can be summarized as: any algebra $B$ satisfying
 \begin{align*}
\gldim B\leq n+1\leq \domdim B
\end{align*}
can be obtained as $B=\End_{\Lambda}(\cM)$ where $\cM$ is $n$-cluster tilting subcategory of $\modd\Lambda$ for an algebra $\Lambda$.

\section{Endomorphism Algebra of The Category $\cS^k$}\label{section dynkin}
 Let $\modd\Lambda$ be the category of finitely generated modules of the algebra $\Lambda=\mathbb{K}Q$ for a Dynkin quiver Q. Recall that $\add\cP=\cS^k$ where
 \begin{align*}
 \cP=\bigoplus\limits_{X\in\text{Ind-}\Lambda, 0\leq j\leq k} X[j].
 \end{align*}

Let $\Gamma^k$ be the endomorphism algebra of $\cP$ over $\cS^k$. Then, the functor $\Hom_{\cS^k}(\cP,-):\cS^k\rightarrow \modd\Gamma$ induces an equivalence between $\add\cP$ and projective $\Gamma^k$-modules. So, every projective object in $\modd\Gamma^k$ is of the form $\Hom_{\cS^k}(\cP,X)$ where $X$ is summand of $\cP$. Moreover, we will show that $\Hom_{\cS^k}(\cP,X[j])$ is always a projective-injective object in $\modd\Gamma^k$. 

We modify proofs of lemmas of \cite[5.2,5.3 VI]{auslander1997representation} for our set up. 
\begin{remark} We denote $X[0]\in\cS^k$ by only $X$. Moreover, by the transparent structure of the derived category of hereditary algebras, we do not distinguish stalk complex at $X[0]$ and $X\in\modd\Lambda$ by abuse of notation.
\end{remark}

\begin{proposition}\label{prop Sk projective resolution, gldim general}
Let $Y$ be in $\modd\Gamma^k$, $k\geq 1$. Then, we have the following. \\
(a) Suppose $P_1\xrightarrow{f} P_0\rightarrow 0$ is a projective $\Gamma^k$-presentation for $Y$. 
Then, there exists  $M_1\xrightarrow{g} M_0$ in $\cS^k$ such that $\Hom_{\cS^k}(\cP,M_1)\cong P_1$, $\Hom_{\cS^k}(\cP,M_0)\cong P_0$ and $\Hom_{\cS^k}(\cP,g)\cong f$.\\
(b) $\pdim Y\leq 3k+2$.
\end{proposition}

\begin{proof}
a) Let $P_1\xrightarrow{f} P_0\rightarrow Y$ be projective presentation of $\Gamma^k$-module $Y$. Since $\Hom_{\cS^k}(\cP,-): \cS^k\rightarrow \modd\Gamma^k$ induces an equivalence between $\add\cP$ and projective $\Gamma^k$-modules, there is a morphism $g: M_1\rightarrow M_0$ in $\cS^k$ which induces $f$. \\

b) Notice that $\Hom_{\cS^k}(X[i],X[j])=0$ if $i<j-1$.
Let $M_2,M_1,M_0\in G[\leq j]$ where $G[\leq j]:=\bigoplus_{j'\leq j} G[j']$.

It is enough to show that for any $Z\in\modd\Gamma^1$, $\pdim Z\leq 5$. Because, if 
\begin{align*}
\cdots\rightarrow\Hom_{\cS^k}(\cP,M_2)\rightarrow \Hom_{\cS^k}(\cP,M_1)\rightarrow \Hom_{\cS^k}(\cP,M_0)\rightarrow Y
\end{align*}
is the projective resolution of $Y\in\modd\Gamma^k$, then we can take the sequence
\begin{align*}
M_2[-1]\rightarrow M_1[-1]\rightarrow M_0[-1]\rightarrow M_2\rightarrow M_1\rightarrow M_0
\end{align*}
in $\cS^k$, so that $\pdim Y=\pdim Y'+3$ where $\Hom_{\cS^k}(\cP,M_1[-1])\rightarrow \Hom_{\cS^k}(\cP,M_0[-1])\rightarrow Y'\rightarrow 0$ is projective presentation of $Y'$. To get an upper bound for projective dimension, it is enough to consider $M_2,M_1,M_0\in(\modd\Lambda)[k]$. In this case,  $\pdim_{\Gamma^k}Y=3(k-1)+\pdim_{\Gamma^{1}} Z$. Now we show that $\pdim_{\Gamma^1} Z\leq 5$. To get an upper bound it is enough to take $M_2,M_1,M_0\in\modd\Lambda[1]$, since the terms from $\modd\Lambda$ cannot increase the projective dimension. We get the sequence
\begin{align*}
0\rightarrow M_2[-1]\rightarrow M_1[-1]\rightarrow M_0[-1]\rightarrow M_2\rightarrow M_1\rightarrow M_0
\end{align*}
in $\cS^1$. By lemma \ref{lemma hom functor is exact on triangles}, we get $\pdim_{\Gamma^1} Z\leq 5$. Hence $\pdim Y\leq 3k+2$ for any $\Gamma^k$ module.
\end{proof}

\begin{proposition}\label{prop gamma k global dimension} Global dimension of $\Gamma^k$ is $3k+2$
\end{proposition}
\begin{proof}
Let $S$ be a simple $\Lambda$-module with nonsimple projective cover $P(S)$, then the sequence
\begin{align*}
0\rightarrow\Omega^1(S)\rightarrow P(S)\rightarrow S\rightarrow \Omega^1(S)[1]\rightarrow\cdots\rightarrow S[k]\rightarrow 0
\end{align*}
applied $\Hom_{\cS^k}(\cP,-)$ gives the projective resolution of $Y\in\modd\Gamma^k$, such that $\pdim Y\geq 3k+2$. Combining with Proposition \ref{prop Sk projective resolution, gldim general} gives the result.
\end{proof}

 Our aim now is to show that $\domdim\Gamma^k=3k+2$. 

\begin{proposition}\label{prop gamma k injectivity} Let $\cP$ be the additive generator for $\cS^k$. 
\begin{enumerate}[label=\arabic*)]
\item  $\Gamma^k$-modules of the form $\Hom_{\cS^k}(\cP,N)$ where $N$ is either injective $\Lambda$-module or $N=X[j]$ for some $1\leq j\leq k$ are injective.
\item A $\Gamma^k$-module is a projective-injective module if and only if it is isomorphic to $\Hom_{\cS^k}(\cP,N)$ for either some injective $\Lambda$-module $N$ or any shifted object $N=X[j]$.
\item The functor $\Hom_{\cS^k}(\cP,-):\cS^k\rightarrow\modd\Gamma^k$ induces an equivalence between the injective and shifted objects of $\cS^k$ and the category of projective-injective $\Gamma^k$-modules.
\end{enumerate}
\end{proposition}

\begin{proof} 
1) First, we give the proof for $I$ is an injective $\Lambda$-module. By definition \ref{definition dynkin category}, $\Hom_{\cS^k}(G[j],I)=0$ for any $j\geq 1$ where $\add G:=\modd\Lambda$. Therefore,
\begin{align}\label{eq isomorphism in cokernel of auslander algebra}
\Hom_{\cS^k}(\cP,I)\cong \Hom_{\cS^k}(G,I)\cong\Hom_{\Lambda}(G,I).
\end{align}
By \cite[Lemma 5.3]{auslander1997representation}, it is an injective object.\\

As we stated in the remark \ref{remark serre functor on triangulated cat}, we have the functorial isomorphism 
\begin{align}\label{eq in derived equivalence}
\Hom_{D^b(\modd\Lambda)}(-,(\tau X)[1])\cong D\Hom_{D^b(\modd\Lambda)}(X,-)
\end{align}
\cite{happel1988triangulated}, \cite{keller2005triangulated}. If we apply \ref{eq in derived equivalence} for $1\leq j\leq k$, we get
\begin{align*}
\Hom_{\cS^k}(\cP,X[j])&\cong \Hom_{D^b(\modd\Lambda)}(\cP,X[j])\\
&\cong D\Hom_{D^b(\modd\Lambda)}(\tau^{-1}X[j-1],\cP)\\
&\cong D\Hom_{\cS^k}(\tau^{-1}X[j-1],\cP)
\end{align*}

$\tau^{-1}X[j-1]$ is projective object in $D^b(\modd\Lambda^{op})$ if and only if $\tau^{-1}X[j-1]\in\cP$. Since $j\geq 1$, its dual is injective. Notice that this argument cannot work for $I[j]$ since $\tau^{-1}I=0$. Nevertheless, we can deduce that:
\begin{align*}
\Hom_{\cS^k}(\cP,I[j])&\cong\Hom_{\Lambda}(G,I)\oplus\ext^1_{\Lambda}(G,I)\\
&\cong\Hom_{\Lambda}(G,I)
\end{align*}
since $\ext^i_{\Lambda}(G,I)=0$ for any injective object $I$. Now we can use Nakayama functor to get
\begin{align*}
\Hom_{\Lambda}(G,I)\cong D\Hom_{\Lambda}(\nu^{-1}I,G)
\end{align*}
Since $\nu^{-1}I$ is projective and summand of $G$, $\Hom_{\Lambda}(\nu^{-1}I,G)$ is projective object over $\Lambda^{op}$. Hence, its dual is injective.\\
2) First we analyze the case restricted to $\modd\Lambda$. Let $P$ be a projective-injective $\Gamma^k$ module. Since $P$ is projective, there exists $X\in\modd\Lambda$ so that $P\cong\Hom_{\Lambda}(\cP,X)$. Let $X\rightarrow I$ be $\Lambda$ injective envelope of $X$. Since $\Hom_{\Lambda}(\cP,X)$ is injective, the monomorphism $\Hom_{\Lambda}(\cP,X)\rightarrow \Hom_{\Lambda}(\cP,I)$  of $\Gamma^k$ modules splits. This means that $X\rightarrow I$ splits. Hence the monomorphism $X\rightarrow I$ is an isomorphism since it is an essential split monomorphism. Thus, we get $P\cong\Hom_{\Lambda}(\cP,I)$\\
For $P=\Hom_{\cS^k}(\cP,X[j])$, in 1) we showed that $\tau^{-1}X[j-1]$ is projective object in $D^b(\modd\Lambda^{op})$ if and only if $\tau^{-1}X[j-1]\in\cP$. As a result either $P\cong\Hom_{\cS^k}(\cP,I)$ or $P\cong\Hom_{\cS^k}(\cP,X[j])$ is projective-injective.\\
3) This is consequence of 2).
\end{proof}

\begin{lemma}\label{lemma minimal inj res DYNKIN} Let $M$ be an indecomposable non-injective $\Lambda$ module. Then 
\begin{align*}
0\rightarrow\Hom_{\cS^k}(\cP, M)\rightarrow \Hom_{\cS^k}(\cP,I_0)\xrightarrow{g}\Hom_{\cS^k}(\cP, I_1)\rightarrow \Hom_{\cS^k}(\cP,M[1])\rightarrow\cdots \Hom_{\cS^k}(\cP,I_1[k])
\end{align*}
is the minimal injective coresolution of $\Hom_{\cS^k}(\cP,M)$ in $\modd\Gamma^k$.
\end{lemma}
\begin{proof}
Let $0\rightarrow M\rightarrow I_0\rightarrow I_1$ be injective copresentation of $M$. Since $\Lambda$ is hereditary, it is the injective coresolution. Consider the map $\coker g[j]\xrightarrow{f[j]} \Hom_{\cS^k}(\cP,M[j+1])$. Assume to the contrary that it is not left minimal. By dual statement of \cite[Cor. 2.3]{auslander1997representation}, this is equivalent to $\im f[j]\cap Z=0$ for $\Hom_{\cS^1}(\cP,M[j+1])\cong Y\oplus Z$ in $\modd\Gamma^k$. Moreover $\im f[j]\subset\Hom_{\cS^k}(\cP,I_0[j+1])$, therefore it is induced by the embedding $M\rightarrow I_0$. Hence, If $Z\neq 0$, then the sequence is not exact which contradicts lemma \ref{cor to lemma hom is exact on triangles}.
\end{proof}

\iffalse

\fi

We restate Theorem \ref{THM dynkin} and give its proof.

\begin{theorem}
The algebra $\Gamma^k:=\End_{\cS^k}\left(\cP\right)$ is a higher Auslander algebra of global dimension $3k+2$ for $k\geq 1$. The opposite quiver of $\Gamma^k$  is equal to the Auslander-Reiten quiver of $\cS^k$. 
\end{theorem}
\begin{proof}
In proposition \ref{prop Sk projective resolution, gldim general} we showed that $\gldim\Gamma^k=3k+2$. We need to compute its dominant dimension.\\
Let $0\rightarrow G\rightarrow I_0\xrightarrow{g} I_1\rightarrow 0$ be the injective coresolution of $\add G=\modd\Lambda$. It is exact, since $\Lambda$ is hereditary. It induces the sequence

\begin{align}
0\rightarrow G\rightarrow I_0\rightarrow I_1\rightarrow G[1]\rightarrow\cdots\rightarrow I_1[k]\rightarrow 0
\end{align}
in $\cS^k$. By corollary \ref{cor to lemma hom is exact on triangles}, and proposition \ref{prop gamma k injectivity}
\begin{align*}
0\rightarrow \Hom_{\cS^k}(\cP,G)\rightarrow \Hom_{\cS^k}(\cP,I_0)\rightarrow \Hom_{\cS^k}(\cP,I_1)\rightarrow \Hom_{\cS^k}(\cP,G[1])\rightarrow\cdots\rightarrow \Hom_{\cS^k}(\cP,I_1[k])
\end{align*}
is the injective $\Gamma^k$ resolution.
 We will show that the cokernel of the map $\Hom_{\cS^k}(\cP,I_0[k])\xrightarrow{f}\Hom_{\cS^k}(\cP,I_1[k])$ is an injective $\Gamma^k$ module. Consider the diagram
 
 \begin{align}
 \xymatrix{
 \Hom_{\cS^k}(\cP,I_0[k])\ar[d]^{\cong}\ar[r]^{f}&\Hom_{\cS^k}(\cP,I_1[k])\ar[r]\ar[d]^{\cong} &\coker f\ar@{..>}[d]\ar[r]&0\\
  \Hom_{\Lambda}(G,I_0)\ar[r]^{f'}&\Hom_{\Lambda}(G,I_1)\ar[r]& \coker f'\ar[r]& 0  }
 \end{align}
where we used \ref{eq isomorphism in cokernel of auslander algebra} for vertical isomorphisms and $f':=\Hom_{\Lambda}(G,g)$. Exactness of the first row follows from proposition \ref{prop Sk projective resolution, gldim general}, i.e., $\gldim \Gamma^k=3k+2$. Since $\coker f'$ is an injective object in Auslander algebra of $\modd\Lambda$, and using exactness, we get $\coker f$ as injective $\Gamma^k$ module. We conclude that $\domdim_{\Gamma^k}\Gamma^k=3k+2$.\\
Now we describe the quiver of $\Gamma^k$. Since $\Gamma^k=\End_{\cS^k}(\cP)$, and $\add\cP=\cS^k$, the opposite quiver of $\Gamma^k$ is equal to the Auslander-Reiten quiver of $\cS^k$. Since $\cS^k$ can be expressed as $G\oplus G[1]\oplus\cdots G[k]$ where $\add G=\modd\Lambda$, $\Gamma^k$ contains $(k+1)$ copies of Auslander-Reiten quiver of $\modd\Lambda$.
\end{proof}

\begin{remark} This is a generalization of Auslander algebras of representation finite and hereditary algebras, in the sense that $k=0$ corresponds to Auslander algebra which we started.
\end{remark}

%\begin{remark}\co{CIM paper}. Here we take the whole module category, not the stable part of $\nu$-formal algebras.\end{remark}

\section{Endomorphism Algebra of The category $\cM^k$}\label{section n-rep finite}
We recall the definition of the category $\cM^k$.
\begin{definition}

Let $A$ be an $n$-representation finite and $n$-hereditary algebra where $\cM$ is the unique $n$-cluster tilting object. Consider the full subcategory of the bounded derived category $D^b(\modd A)$ denoted by $\cM^k$ whose objects are
\begin{align}
\left\{M\,\vert\, M=X[jn], 0\leq j\leq k, \forall X\in\cM\right\}.
\end{align}
and whose morphisms are
\begin{align}
\Hom_{\cM^k}(X[in],Y[jn])=\begin{cases}
\Hom_{\Lambda}(X,Y),\,\quad \text{if} \,\quad i=j \\
\ext^n_{\Lambda}(X,Y),\,\,\,\quad \text{if}\,\quad i=j-1\\
0\,\quad \text{otherwise}.
\end{cases}
\end{align}
\end{definition}
Let $\Gamma^k:=\End_{\cM^k}(\cP)$, where $\add\cP=\cM^k$. Since $\cM^k$ is full subcategory of $\cM[n\mathbb{Z}]$,  $\Gamma^k$ can be expressed as $\End_{D^b(\modd A)}(\cP)$. Let $\add G=\cM$ where $G\in\modd A$ and $G[\leq jn]$ be $\bigoplus_{1\leq j'\leq j}G[j'n]$.

\begin{proposition}\label{prop global dimension of higher algebra}
Let $\cP$ be an 
additive generator of $\cM^k$ and let $Y$ be in $\modd\Gamma^k$. Then we have the following. \\
(a) Suppose $P_1\xrightarrow{f} P_0\rightarrow 0$ is a projective $\Gamma^k$-presentation for $Y$. 
Then there exists  $M_1\xrightarrow{g} M_0$ in $\cM^k$ such that $\Hom_{\cM^k}(\cP,M_1)\cong P_1$, $\Hom_{\cM^k}(\cP,M_0)\cong P_0$ and $\Hom_{\cM^k}(\cP,g)\cong f$.\\
(b) $\pdim Y\leq (n+2)k+n+1$.
\end{proposition}

\begin{proof}
(a) Since $\Hom_{\cM^k}(\cP,-)$ induces an equivalence between $\add\cP$ and projective $\Gamma^k$ modules, there is a morphism $M_1\xrightarrow{g} M_0$ in $\cM^k$ such that the induced morphism $\Hom_{\cM^k}(\cP,g)$ is isomorphic to $f$.\\
(b) There are three possibilities we analyze.
\begin{enumerate}
[label=\arabic*)]
\item If $M_1,M_0\in\cM$, then $\ker g$ has an approximation by \cite[Prop 2.3]{iyama2008auslanderreitenrevisited}. Hence $\pdim Y\leq n+1$.
\item If $M_1\in \cM[jn]$ and $M_0\in \cM[(j+1)n]$, then there exists an $n$-exact sequence
\begin{align*}
0\rightarrow M'_0\rightarrow \cdots \rightarrow M'_1\rightarrow 0
\end{align*} 
in the sense of \cite{jasso2016n}
which induces the $(n+2)$-angle
\begin{align*}
M'_0\rightarrow \cdots \rightarrow M'_1\xrightarrow{g} M'_0[n]
\end{align*} where $M'_1[jn]=M_1$, $M'_0[jn]=M_0$.
Moreover all left rotations upto $[kn]$ are in $\cM^k$. If we apply $\Hom_{\cM^k}(\cP,-)$, we get $\pdim Y\leq (n+2)(j-1)+n+2$ for any $1\leq j\leq k$ by lemma \ref{lemma hom functor is exact on n angles higher version} since the summands of $M_1$ and $M_0$ which belong to $G[\leq jn]$ cannot increase the projective dimension.
\item If $M_1,M_0\in \cM[jn]$, then there exists morphism $M'_1\xrightarrow{g'} M'_0$ where $M'_1[jn]=M_1$, $M'_0[jn]=M_0$,$g'[jn]=g$. There is an approximation of $\ker g'$ in $\cM$, i.e.
\begin{align*}
0\rightarrow N_{n+1}\rightarrow N_n\rightarrow\cdots\rightarrow\ker g'
\end{align*}
where $N_{n+1}\rightarrow N_n$ is monomorphism. By \cite[Axiom 3]{jasso2016n}, there exists $n$-exact sequence
\begin{align*}
0\rightarrow N_{n+1}\rightarrow N_n\rightarrow\cdots\rightarrow N_0\rightarrow 0
\end{align*}
in $\cM$, which induces morphism $N_0\rightarrow N_{n+1}[n]$. Therefore $\pdim Y\leq (n+2)(j-1)+n+1+\pdim_{A}\ker g'$. By part 1), claim holds. We use the same argument that the summands of $M_1,M_0$ which belong to $G[\leq (j-1)n]$ cannot increase the projective dimension.
\end{enumerate}
\end{proof}

\begin{proposition}\label{prop projective-injective modules Higher} Let $\cP$ be the additive generator of $\cM^k$. 
\begin{enumerate}[label=\arabic*)]
\item  $\Gamma^k$-modules of the form $\Hom_{\cM^k}(\cP,N)$ where $N$ is either injective $\cM$-module or $N=X[jn]$ for some $1\leq j\leq k$ is injective.
\item A $\Gamma^k$ module is a projective-injective module if and only if it is isomorphic to $\Hom_{\cM^k}(\cP,N)$ for either some injective $\cM$-module $N$ or any shifted object $N=X[jn]$.
\item The functor $\Hom_{\cM^k}(\cP,-):\cM^k\rightarrow\modd\Gamma^k$ induces an equivalence between the injective and shifted by multiples of n objects of $\cM^k$ and the category of projective-injective $\Gamma^k$-modules.
\end{enumerate}
\end{proposition}

\begin{proof}
1) This follows from lemmas \ref{lemma HIGHER injective object via injective} and \ref{lemma HIGHER injective via shifts} below.\\
2) If $P$ is of the form $\Hom_{\cM^k}(\cP,I)$ where $I\in \modd A$ is injective, then by \cite[Lemma 4.1]{iyama2011clusterMAINpaper}, it follows. If $P$ is isomorphic to $\Hom_{\cM^k}(\cM,X[jn])$, then it follows from lemma \ref{lemma HIGHER injective via shifts}.\\
3) This is consequence of 2).
\end{proof}

\begin{lemma}\label{lemma HIGHER injective object via injective} $\Hom_{\cM^k}(\cP,I[jn])$ for any $0\leq j\leq k$ is injective $\Gamma^k$ module.
\end{lemma}
\begin{proof}
Let $\add G=\cM$. If $j=0$, then $\Hom_{\cM^k}(\cP,I)\cong\Hom_{A}(G,I)$ is injective by \cite[Lemma 4.1]{iyama2011clusterMAINpaper}. For $j\geq 1$, we have
\begin{align*}
\Hom_{\cM^k}(\cP,I[jn])&\cong \Hom_{\cM^k}(G[jn],I[jn])\\
&\cong\Hom_{A}(G,I)\oplus\ext^n_{A}(G,I)\\
&\cong\Hom_{A}(G,I)
\end{align*}
is injective where we used the facts that $\ext^n_{A}(G,I)=0$ by injectivity of $I$ and $\Hom_{\cM^k}(G[in],I[jn])=0$ for any $i<j-1$ by definition \ref{def the higher category M_k}.
\end{proof}

\begin{lemma}\label{lemma HIGHER injective via shifts} Let $X\in\cM$ be non-injective module. Then, $\Hom_{\cM^k}(\cP,X[jn])$ is injective $\Gamma^k$ module for $1\leq j\leq k$.
\end{lemma}

\begin{proof}

By remark \ref{remark serre functor on higher angulated cat}, there is a functorial isomorphism
\begin{align}
\Hom_{D^b(\modd A)}(-,(\tau_n X)[n])\cong D\Hom_{D^b(\modd A)}(X,-)
\end{align}

Hence, we get
\begin{align*}
\Hom_{\cM^1}(\cP,X[jn])&\cong \Hom_{D^b(\modd A)}(\cP,X[jn])\\
&\cong D\Hom_{D^b(\modd A)}((\tau^{-1}_n X)[(j-1)n],\cP)\\
&\cong D\Hom_{\cM^k}((\tau^{-1}_n X)[(j-1)n],\cP)
\end{align*}
Since $\tau_n^{-1}X[(j-1)n]\in\cP$, it is projective object in $D^b(\modd\Lambda^{op})$. Therefore, its dual is injective.
\end{proof}

\begin{proposition}\label{prop higher algebra is higher auslander computation for domdim} The algebra $\Gamma^k=\End_{\cM^k}(\cP)$ is a higher Auslander algebra of global dimension $(n+2)k+n+1$.
\end{proposition}
\begin{proof}
First, we compute global dimension of $\Gamma^k$. Let 
\begin{align}
0\rightarrow M_{n+1}\rightarrow\cdots\rightarrow M_1\rightarrow M_0\rightarrow 0
\end{align}
be an exact sequence in $\cM$. It induces the sequence 
\begin{align}
0\rightarrow M_{n+1}\rightarrow\cdots\rightarrow M_0\rightarrow M_{n+1}[n]\rightarrow \cdots M_0[kn]\rightarrow 0
\end{align}
in $\cM^k$. If we apply $\Hom_{\cM^k}(\cP,-)$, gives the projective resolution of some $\Gamma^k$ module $Y$, such that $\pdim Y\geq (n+2)k+n+1$. Combining with proposition \ref{prop global dimension of higher algebra} gives the result.\\
For any non-injective object $M\in\cM$, consider its injective resolution
\begin{align}
0\rightarrow M\rightarrow I_0\rightarrow I_1\rightarrow\cdots\rightarrow I_{n-1}\xrightarrow{g} I_n\rightarrow 0
\end{align}
$\Hom_{\cM^1}(\cP,-)$ is exact on it by  \ref{lemma hom functor is exact on n angles higher version}. All objects are projective-injective except the first term.
Consider the cokernel of the map $f$ where
\begin{align*}
\Hom_{\cM^k}(\cP,I_{n-1}[kn])\xrightarrow{f}\Hom_{\cM^k}(\cP,I_n[kn]).
\end{align*}
Since global dimension of $\Gamma^k$ is $(n+2)k+n+1$, we get right exact sequence
\begin{align*}
\Hom_{\cM^k}(\cP,I_{n-1}[kn])\xrightarrow{f}\Hom_{\cM^k}(\cP,I_n[kn])\rightarrow \coker f\rightarrow 0.
\end{align*}
On the other hand, $\Hom_{\cM^k}(\cP,I_i[kn])\cong\Hom_{A}(G,I_i)\oplus\ext^n_{A}(G,I_i)=\Hom_{A}(G,I_i)$ since $\ext^n(G,I_i)=0$ by injectivity of $I_i$, we get the diagram

\begin{align*}
\xymatrix{
\Hom_{\cM^k}(\cP,I_{n-1}[kn])\ar[d]^{\cong}\ar[r]^{f}&\Hom_{\cM^k}(\cP,I_n[kn])\ar[r] \ar[d]^{\cong}&\coker f\ar@{..>}[d]\ar[r] &0\\
\Hom_{A}(G,I_{n-1})\ar[r]^{f}&\Hom_{A}(G,I_n)\ar[r]& \coker f'\ar[r] &0}
\end{align*}
where $\coker f'$ is injective in $\modd\End_{A}(G)$. Diagram commutes, hence $\coker f$ is injective in $\modd\Gamma^k$.
Therefore, for any $M\in\cM$, non-injective object, dominant dimension of $\Hom_{\cM^k}(\cP,M)$ is at least $(n+2)k+n+1$. This shows $\Gamma^k$ is a higher Auslander algebra.
\end{proof}

We restate Theorem \ref{THM higher hereditary} and give a proof.

\begin{theorem}
The algebra $\Gamma^k:=\End_{\cM^k}\left(\cP\right)$ is a higher Auslander algebra of global dimension $(n+2)k+n+1$, where $\add\cP=\cM^k$. The opposite quiver of $\Gamma^k$  is equal to the Auslander-Reiten quiver of $\cM^k$ in $\cM[n\mathbb{Z}]\subset D^b(\modd A)$.
\end{theorem}

\begin{proof}
By proposition \ref{prop higher algebra is higher auslander computation for domdim}, $\gldim\Gamma^k\leq (n+2)k+n+1\leq \domdim\Gamma^k$ which shows it is a higher Auslander algebra. 
Now we prove the second statement. By definition $\Gamma^k=\End_{\cM^k}(\cP)$ where $\add\cP=\cM^k$, the opposite quiver of $\Gamma^k$ is simply the Auslander-Reiten quiver of $\cM^k$. Since $\cM^k$ can be expressed as $G\oplus G[n]\oplus\cdots G[kn]$ where $\add G=\cM$, $\Gamma^k$ contains $(k+1)$ copies of Auslander-Reiten quiver of $\cM$.
\end{proof}

\begin{remark} This is a generalization of higher Auslander algebras obtained as endomorphism algebras of $n$-cluster tilting object of $n$-representation finite and $n$-hereditary algebras, in the sense that $k=0$ corresponds to the higher Auslander algebra which we started.
\end{remark}

\section{$\Sigma^k$ is $d$-representation finite}\label{section endomorphism algebras of gamma}
We divide this section into two subsections. First we consider Dynkin quiver case and give proof of Theorem \ref{THM sigma k algebra} 1). In \ref{subsection sigma higher}, we present proof of Theorem \ref{THM sigma k algebra} 2).

\subsection{Dynkin Case}

 Let $Aus(Q)$ be the rank of Auslander algebra of $\mathbb{K}Q$ where rank of $Q$ is $n$.
Let $\Sigma^k=\End_{\Gamma^k}(\cQ)$, $\add G=\modd\Lambda$, $\add\cP=\cS^k$, $\cP=\bigoplus_{0\leq j\leq k}G[j]$.

\begin{lemma}\label{lemma endo of sigma DYNKIN}
The algebra $\Sigma^k$ is isomorphic to $\End_{\cS^k}(D\Lambda\oplus \bigoplus_{1\leq j\leq k}G[j])$.
\end{lemma}

\begin{proof}
In proposition \ref{prop gamma k injectivity}, we showed that any projective-injective $\Gamma^k$ module is of the form $\Hom_{\cS^k}(\cP,I)$ or $\Hom_{\cS^k}(\cP,X[j])$ where $I\in\modd\Lambda$ is injective and $1\leq j\leq k$. Therefore $\cQ=\Hom_{\cS^k}(\cP,D\Lambda\oplus\bigoplus_{1\leq j\leq k}G[j] )$.
We have 
\begin{align*}
\Sigma^k&=\End_{\Gamma^k}(\cQ)\cong\Hom_{\Gamma^k}(\cQ,\cQ)\\
&\cong \Hom_{\Gamma^k}\left(\Hom_{\cS^k}(\cP,D\Lambda\oplus\bigoplus_{1\leq j\leq k}G[j] ),\Hom_{\cS^k}(\cP,D\Lambda\oplus\bigoplus_{1\leq j\leq k}G[j] )\right)\\
&\cong \Hom_{\cS^k}(D\Lambda\oplus\bigoplus_{1\leq j\leq k}G[j],D\Lambda\oplus\bigoplus_{1\leq j\leq k}G[j])\\
&=\End_{\cS^k}(D\Lambda\oplus\bigoplus_{1\leq j\leq k}G[j])
\end{align*}
by Yoneda's lemma.
\end{proof}

\begin{remark}\label{remark when sigma k is projective-injective} We give characterization when $\Hom_{\cS^k}(D\Lambda \oplus \bigoplus_{1\leq j\leq k}G[j], Z)$ is projective non-injective $\Sigma^k$ module. 
Let $Z\in\modd\Lambda$ be a non-injective object. Then by remark \ref{remark serre functor on triangulated cat}, we have
\begin{align*}
\Hom_{\cS^k}(D\Lambda \oplus \bigoplus_{1\leq j\leq k}G[j], Z)&\cong\Hom_{D^b(\modd\Lambda)}(D\Lambda \oplus \bigoplus_{1\leq j\leq k}G[j], Z)\\
&\cong D\Hom_{D^b(\modd\Lambda)}(\tau^{-1}Z[-1],D\Lambda \oplus \bigoplus_{1\leq j\leq k}G[j])\\
&\cong D\Hom_{\cS^k}(\tau^{-1}Z[-1],D\Lambda \oplus \bigoplus_{1\leq j\leq k}G[j])
\end{align*}
Therefore $\Hom_{\cS^k}(D\Lambda \oplus \bigoplus_{1\leq j\leq k}G[j], Z)$ is injective if and only if $\tau^{-1}Z[-1]\in D\Lambda \oplus \bigoplus_{1\leq j\leq k}G[j]) $. This implies, either $Z=X[j]$ for $j\geq 2$ or $Z=X[1]$ with $\tau^{-1}X\in D\Lambda$. \\
Similarly, for an injective $I\in\modd\Lambda$, the functorial isomorphism becomes
\begin{align*}
\Hom_{\cS^k}(D\Lambda \oplus \bigoplus_{1\leq j\leq k}G[j], I)&\cong\Hom_{\Lambda}(D\Lambda, I)\\
&\cong D\Hom_{\Lambda}(\nu^{-1}I,D\Lambda)
\end{align*}
Therefore $\Hom_{\Lambda}(\nu^{-1}I,D\Lambda)$ is projective over $\modd\Lambda^{op}$ if and only if $\nu^{-1}I\in D\Lambda$. 
\end{remark}
\begin{lemma}\label{lemma number of proj non injective DYKIN} The number of projective non-injective $\Sigma^k$ modules is $Aus(Q)-n$.
\end{lemma}
\begin{proof}
By remark \ref{remark when sigma k is projective-injective}, any $\Hom_{D^b(\modd\Lambda)}(D\Lambda \oplus \bigoplus_{1\leq j\leq k}G[j],X[j])$, $j\geq 2$ is projective-injective $\Sigma^k$ module. For the remaining objects, the closure of $\tau^{-i}X$ orbit where $i\geq 2$ cannot belong to $D\Lambda$, hence there are $Aus(Q)-n$ many projective non-injective $\Sigma^k$ modules. By duality, it is the same number of injective non-projective $\Sigma^k$ modules.
\end{proof}

 We summarize the results below.

\begin{proposition} We have the following cardinalities.
\begin{enumerate}[label=\arabic*)]
\item The rank of $\Gamma^k$ is $(k+1).Aus(Q)$
\item The rank of $\Sigma^k$ is $k.Aus(Q)+n$
\item The number of projective-injective objects of $\modd\Gamma^k$ is the rank of $\Sigma^k$.
\item The number of projective-injective objects of $\Sigma^k$ is $(k-1)Aus(Q)+2n$.
\item The number of injective but not projective $\Gamma^k$ modules is $Aus(Q)-n$.
\item The number of injective but not projective $\Sigma^k$ modules is $Aus(Q)-n$.
\end{enumerate}
\end{proposition}
\begin{proof}
\begin{enumerate}[label=\arabic*)]
\item Since $\cS^k$ has $(k+1)Aus(Q)$ indecomposable objects, claim follows.
\item The rank of $\Sigma^k$ is the number of projective-injective objects of $\Gamma^k$. Every projective object in the subquiver $\Gamma^{k-1}$ is projective-injective. The subquiver of the Auslander algebra part has $n$ projective-injective objects. Therefore $n+kAus(Q)$ is the rank of $\Sigma^k$.
\item Since $\Sigma^k$ is the endomorphism algebra of projective-injective objects of $\Gamma^k$, The rank is the number of nonisomorphic projective-injective $\Gamma^k$-modules
\item By remark \ref{remark when sigma k is projective-injective} and lemma \ref{lemma number of proj non injective DYKIN}
\item The injective but not projective objects lie in the Auslander algebra part. There are n projective-injective objects there. Hence in total, there are $Aus(Q)-n$ many injective nonprojective objects of $\modd\Gamma^k$.
\item By remark \ref{remark when sigma k is projective-injective} and lemma \ref{lemma number of proj non injective DYKIN}
\end{enumerate}
\end{proof}

\begin{theorem}\label{thm sigma is d rep DYNKIN}
Let $d=3k+1$. Then, $\Sigma^k$ is $d$-representation finite algebra.
\end{theorem}

\begin{proof}
By proposition \ref{prop gamma k global dimension}, $\Sigma^k$ has a $d$-cluster tilting object $\cC$. By \cite[Prop 1.5]{iyama2011clusterMAINpaper}, $\cC$ should contain the $\tau_d$ closure of $D\Sigma^k$. In particular it should contain $\Sigma^k$.  Notice that the rank of $\cC$ is $(k+1)Aus(Q)$, and the rank of projective-injective objects of $\Sigma^k$ is $(k-1)Aus(Q)+2n$. If we add projective non-injective together with injective non-projective $\Sigma^k$ modules, we get
\begin{align}
\rank\Gamma^k=\#\cC\geq (k-1)Aus(Q)+2n+2(Aus(Q)-n)=(k+1)Aus(Q). 
\end{align}
So, $\add(D\Sigma^k\oplus \Sigma^k)$ is $d$-cluster tilting subcategory in $\modd\Sigma^k$. By Theorem 1.6 of \cite{iyama2011clusterMAINpaper}, it is enough to prove that $\gldim\Sigma^k=3k+1$ which we show in proposition \ref{prop global dimension of sigma k DYNKIN}. Hence $\Sigma^k$ is $d$-representation finite algebra for any $d=3k+1$, $k\geq 1$.
\end{proof}

\begin{proposition}\label{prop global dimension of sigma k DYNKIN} The global dimension of $\Sigma^k$ is $3k+1$.
\end{proposition}
\begin{proof}
Since $\Lambda$ is representation finite and hereditary algebra, the injective resolution
in $0\rightarrow M\rightarrow I_0\rightarrow I_1\rightarrow 0$  is exact and induces the sequence 
\begin{align*}
0\rightarrow M\rightarrow I_0\rightarrow I_1\rightarrow M[1]\rightarrow\cdots\rightarrow I_1[k]\rightarrow 0
\end{align*}
in $\cS^k$. By lemma \ref{lemma endo of sigma DYNKIN}, the sequence
\begin{align*}
0\rightarrow I_0\rightarrow I_1\rightarrow M[1]\rightarrow\cdots\rightarrow I_1[k]\rightarrow 0
\end{align*}
is in $D\Lambda\oplus \bigoplus G[j]$ which we denote  its additive generator by $\cG$. If we apply $\Hom_{D^b(\modd\Lambda)}(\cG,-)$ to the sequence above, we get
\begin{align*}
0\rightarrow \Hom_{D^b(\modd\Lambda)}(\cG, I_0)\rightarrow \Hom_{D^b(\modd\Lambda)}(\cG,I_1)\rightarrow\Hom_{D^b(\modd\Lambda)}(\cG, M[1])\rightarrow\\\cdots\rightarrow\Hom_{D^b(\modd\Lambda)}(\cG, I_0[k])\rightarrow\Hom_{D^b(\modd\Lambda)}(\cG, I_1[k])
\end{align*}
which gives the projective resolution of $Y\in\modd\Sigma^k$, which is $3k+1$, since $0\rightarrow\Hom(\cG,I_0)\rightarrow \Hom(\cG,I_1)$ is monomorphism and $\Hom(\cG,I_0)$ is projective non-injective $\Sigma^k$ module by \ref{remark when sigma k is projective-injective}. 
\end{proof}

\begin{proposition} Cluster tilting object of $\Sigma^k$ is $\add(D\Sigma^k\oplus\tau_{3k+1}D\Sigma^k)\cong \add(\Sigma^k\oplus D\Sigma^k)$ is minimal generator-cogenerator of $\Sigma^k$.
\end{proposition}
\begin{proof}
By Theorem \ref{thm sigma is d rep DYNKIN}, $\Sigma^k$ is $d$-representation finite algebra of global dimension $d$. Therefore higher Auslander-Reiten translate $\tau_{d}$ exists. By \cite[Theorem 1.6]{iyama2011clusterMAINpaper}, result follows.
\end{proof}

\begin{proposition} $\Sigma^k$ can be realized as endomorphism algebra of fundamental domain of k-cluster category.
\end{proposition}
\begin{proof}
Recall that the fundamental domain of $k$-cluster category is of the form 
\begin{align*}
G\oplus G[1]\oplus\cdots\oplus G[k-1]\oplus \Lambda[k]
\end{align*}
which is equivalent to
\begin{align*}
D\Lambda[-1]\oplus  G\oplus G[1]\oplus\cdots\oplus G[k-1]
\end{align*}
On the other hand, by using symmetry structure of $D^b(\modd\Lambda)$, we can apply $[-1]$ to $D\Lambda\oplus G[1]\oplus\cdots\oplus G[k]$, and its endomorphsim algebra is Morita equivalent to $\Sigma^k$. Therefore $\Sigma^k$ can be realized as endomorphism algebra of fundamental domain of k-cluster category.
\end{proof}

\begin{corollary}
The quiver of $\Sigma^k$ is simply the Auslander-Reiten quiver of fundamental domain of $k$-cluster category. 
\end{corollary}

\begin{remark} We remark that $Aus(Q)=\vert R^{+}\vert$ where $R^+$ is set of all positive roots of $Q$ \cite{GLS06}. Therefore the cardinality of $(3k+1)$-cluster tilting object of $\Sigma^k$ is a multiple of $\vert R^{+}\vert$. .
\end{remark}

\subsection{$n$-Representation Finite Case}\label{subsection sigma higher}

Let $R$ be the rank of $n$-representation finite $n$-hereditary algebra $A$ where $\cM\subset\modd A$ is $n$-cluster tilting subcategory. Let $\#\cM$ be the rank of higher Auslander algebra of $\Gamma^k$, $\Sigma^k=\End_{\Gamma^k}(\cQ)$ where $\cQ$ is additive generator of projective-injective $\Gamma^k$ modules, $\add G=\modd A$, $\add\cP=\cM^k$, $\cP=\bigoplus_{0\leq j\leq k}\cM[jn]$.

\begin{lemma}\label{lemma endo of sigma HIGHER}
The algebra $\Sigma^k$ is isomorphic to $\End_{\cM^k}(DA\oplus \bigoplus_{1\leq j\leq k}\cM[jn])$.
\end{lemma}

\begin{proof}
In proposition \ref{prop projective-injective modules Higher}, we showed that any projective-injective $\Gamma^k$ module is of the form $\Hom_{\cM^k}(\cP,I)$ or $\Hom_{\cM^k}(\cP,X[jn])$ where $I\in\modd A$ is injective and $1\leq j\leq k$. Therefore $\cQ=\Hom_{\cM^k}(\cP,DA\oplus\bigoplus_{1\leq j\leq k}\cM[jn] )$.
We have 
\begin{align*}
\Sigma^k&=\End_{\Gamma^k}(\cQ)\cong\Hom_{\Gamma^k}(\cQ,\cQ)\\
&\cong \Hom_{\Gamma^k}\left(\Hom_{\cM^k}(\cP,DA\oplus\bigoplus_{1\leq j\leq k}\cM[jn] ),\Hom_{\cM^k}(\cP,DA\oplus\bigoplus_{1\leq j\leq k}\cM[jn] )\right)\\
&\cong \Hom_{\cM^k}(DA\oplus\bigoplus_{1\leq j\leq k}\cM[jn],DA\oplus\bigoplus_{1\leq j\leq k}\cM[jn])\\
&=\End_{\cM^k}(DA\oplus\bigoplus_{1\leq j\leq k}\cM[jn])
\end{align*}
by Yoneda's lemma.
\end{proof}

\begin{remark}\label{remark when sigma k is projective inj HIGHER} We give characterization when $\Hom_{\cM^k}(DA \oplus \bigoplus_{1\leq j\leq k}\cM[jn], Z)$ is projective non-injective $\Sigma^k$ module. 
Let $Z\in\cM$ be a non-injective object. Then by autoequivalence in remark \ref{remark serre functor on higher angulated cat}, we have
\begin{align*}
\Hom_{\cM^k}(DA \oplus \bigoplus_{1\leq j\leq k}\cM[jn], Z)&\cong\Hom_{D^b(\modd A)}(DA \oplus \bigoplus_{1\leq j\leq k}\cM[jn], Z)\\
&\cong D\Hom_{D^b(\modd A)}(\tau_n^{-1}Z[-n],DA \oplus \bigoplus_{1\leq j\leq k}\cM[jn])\\
&\cong D\Hom_{\cM^k}(\tau^{-1}_n Z[-n],DA \oplus \bigoplus_{1\leq j\leq k}\cM[jn])
\end{align*}
Therefore $\Hom_{\cM^k}(DA \oplus \bigoplus_{1\leq j\leq k}\cM[jn], Z)$ is injective if and only if $\tau_n^{-1}Z[-n]\in D\Lambda \oplus \bigoplus_{1\leq j\leq k}G[j] $. This implies, either $Z=X[jn]$ for $j\geq 2$ or $Z=X[n]$ with $\tau^{-1}_nX\in DA$. \\
Similarly, for an injective $I\in\modd\Lambda$, the duality becomes
\begin{align*}
\Hom_{\cM^k}(DA \oplus \bigoplus_{1\leq j\leq k}\cM[jn], I)&\cong\Hom_{A}(DA, I)\\
&\cong D\Hom_{A}(\nu^{-1}_n I,DA)
\end{align*}
Therefore $\Hom_{A}(\nu^{-1}_nI,DA)$ is projective over $\modd A^{op}$ if and only if $\nu^{-1}I\in DA$.
\end{remark}
\begin{lemma}\label{lemma number of proj non injective NONDYKIN} The number of projective non-injective $\Sigma^k$ modules is $\#\cM-R$.
\end{lemma}
\begin{proof}
By remark \ref{remark when sigma k is projective inj HIGHER}, any $\Hom_{D^b(\modd A)}(DA \oplus \bigoplus_{1\leq j\leq k}\cM[jn],X[jn])$, $j\geq 2$ is projective-injective $\Sigma^k$ module. For the remaining objects, the closure of $\tau^{-i}_nX$ orbit where $i\geq 2$ cannot belong to $DA$, hence there are $\#\cM-R$ many projective non-injective $\Sigma^k$ modules. By duality, it is the same number of injective non-projective $\Sigma^k$ modules.
\end{proof}
\iffalse
\begin{remark} Lemmas \ref{lemma number of proj non injective DYKIN} and \ref{lemma number of proj non injective NONDYKIN} can be proved also by induction. Consider $\End_{D^b(\modd A)}(\cM\oplus A[n])$. Then $\Hom_{D^b(\modd A)}(\cM\oplus A[n],DA)\cong \Hom_{A}(\cM,DA)$ are projective-injective objects. Similarly $\Hom_{D^b(\modd A)}(\cM\oplus A[n],A[n])\cong D\Hom_{D^b(\modd A)}(\tau^{-1}_nA,\cM\oplus A[n])\cong D\Hom_{D^b(\modd A)}(DA,\cM\oplus A[n])$ is injective object since $DA\in\cM$ is projective over $\modd A^{op}$. 
\end{remark}
\fi
 We summarize the results below.

\begin{proposition} We have the following cardinalities.
\begin{enumerate}[label=\arabic*)]
\item The rank of $\Gamma^k$ is $(k+1)\#\cM$
\item The rank of $\Sigma^k$ is $k.\#\cM+R$
\item The number of projective-injective objects of $\modd\Gamma^k$ is the rank of $\Sigma^k$.
\item The number of projective-injective objects of $\Sigma^k$ is $(k-1)\#\cM+2R$.
\item The number of injective but not projective $\Gamma^k$ modules is $\#\cM-R$.
\item The number of injective but not projective $\Sigma^k$ modules is $\#\cM-R$.
\end{enumerate}
\end{proposition}
\begin{proof}
\begin{enumerate}[label=\arabic*)]
\item Since $\cM^k$ has $(k+1)\#\cM$ indecomposable objects, claim follows.
\item The rank of $\Sigma^k$ is the number of projective-injective objects of $\Gamma^k$. By proposition \ref{prop projective-injective modules Higher} there are  $(k-2)\#\cM+R$ projective-injective $\Gamma^k$ modules which is the rank of $\Sigma^k$.
\item Since $\Sigma^k$ is the endomorphism algebra of projective-injective objects of $\Gamma^k$, The rank is the number of nonisomorphic projective-injective $\Gamma^k$-modules
\item By remark \ref{remark when sigma k is projective inj HIGHER} and lemma \ref{lemma number of proj non injective NONDYKIN}
\item The injective but not projective objects lie in the $\End_{A}(\cM)$ part. There are R projective-injective objects there. Hence in total, there are $\#\cM-R$ many injective nonprojective objects of $\modd\Gamma^k$.
\item By remark \ref{remark when sigma k is projective inj HIGHER} and lemma \ref{lemma number of proj non injective NONDYKIN}
\end{enumerate}
\end{proof}

\begin{theorem}\label{thm sigma is d rep HIGHER}
Let $d=(n+2)k+n$. Then, $\Sigma^k$ is $d$-representation finite algebra.
\end{theorem}

\begin{proof}
By proposition \ref{prop global dimension of higher algebra}, $\Sigma^k$ has a $d$-cluster tilting object $\cC$. By \cite[Prop 1.5]{iyama2011clusterMAINpaper}, $\cC$ should contain the $\tau_d$ closure of $D\Sigma^k$. In particular it should contain $A$. Notice that the rank of $\cC$ is $(k+1)\#\cM$, and the rank of projective-injective objects of $\Sigma^k$ is $(k-1)\#\cM+2R$. If we add projective non-injective together with injective non-projective $\Sigma^k$ modules, we get
\begin{align}
\rank\Gamma^k=\#\cC\geq (k-1)\#\cM+2R+2(\#\cM-R)=(k+1)\#\cM. 
\end{align}
So, $\add(D\Sigma^k\oplus \Sigma^k)$ is $d$-cluster tilting subcategory in $\modd\Sigma^k$. By Theorem 1.6 of \cite{iyama2011clusterMAINpaper}, it is enough to prove that $\gldim\Sigma^k=(n+2)k+n$ which we show in proposition \ref{prop global dimension of sigma k HIGHER}. Hence $\Sigma^k$ is $d$-representation finite algebra for any $d=(n+2)k+n$, $k\geq 1$.
\end{proof}

\begin{proposition}\label{prop global dimension of sigma k HIGHER} The global dimension of $\Sigma^k$ is $(n+2)k+n$.
\end{proposition}
\begin{proof}
Since $A$ is $n$-representation finite and $n$-hereditary algebra, the injective resolution
 $0\rightarrow M\rightarrow I_0\rightarrow I_1\rightarrow \cdots$  is exact and induces the sequence 
\begin{align*}
0\rightarrow M\rightarrow I_0\rightarrow I_1 \rightarrow\cdots\rightarrow I_{n+1}[kn]\rightarrow 0
\end{align*}
in $\cM^k$. By lemma \ref{lemma endo of sigma HIGHER}, the sequence
\begin{align*}
 0\rightarrow I_0\rightarrow I_1\rightarrow\cdots\rightarrow I_{n+1}\rightarrow M[n]\rightarrow\cdots\rightarrow I_{n+1}[kn]\rightarrow 0
\end{align*}
is in $DA\oplus \bigoplus \cM[jn]$ which we denote  its additive generator by $\cG$. If we apply $\Hom_{D^b(\modd\Lambda)}(\cG,-)$ to the sequence above, we get
\begin{align*}
0\rightarrow \Hom_{D^b(\modd A)}(\cG, I_0)\rightarrow \Hom_{D^b(\modd A)}(\cG,I_1)\rightarrow\Hom_{D^b(\modd A)}(\cG, I_2)\rightarrow\\\cdots\rightarrow\Hom_{D^b(\modd A)}(\cG, I_n[kn])\rightarrow\Hom_{D^b(\modd A)}(\cG, I_{n+1}[kn])
\end{align*}
which gives the projective resolution of $Y\in\modd\Sigma^k$, which is $(n+2)k+n$, since $0\rightarrow\Hom(\cG,I_0)\rightarrow \Hom(\cG,I_1)$ is monomorphism and $\Hom(\cG,I_0)$ is projective but non-injective $\Sigma^k$ module. 
\end{proof}

\begin{proposition} Higher cluster tilting object of $\Sigma^k$ is $\add(D\Sigma^k\oplus\tau_{(n+2)k+n}D\Sigma^k)\cong \add(\Sigma^k\oplus D\Sigma^k)$ is minimal generator-cogenerator of $\Sigma^k$.
\end{proposition}
\begin{proof}
By Theorem \ref{thm sigma is d rep HIGHER}, $\Sigma^k$ is $d$-representation finite algebra of global dimension $d$. Therefore higher Auslander-Reiten translate $\tau_{d}$ exists. By Theorem 1.6 of \cite{iyama2011clusterMAINpaper}, result follows.
\end{proof}

\begin{proposition} $\Sigma^k$ can be realized as endomorphism algebra of fundamental domain of higher k-cluster category.
\end{proposition}
\begin{proof}
Recall that the fundamental domain of higher $1$-cluster category is of the form 
\begin{align}
\cM\oplus A[n]
\end{align}
which is defined in \cite{oppermann2012higher}. So, it is natural to construct higher $k$-cluster category which is the orbit category, where we identify objects via $\tau^{-1}_n[kn]$. In this case, the fundamental domain becomes
\begin{align*}
\cM\oplus \cM[n]\oplus\cdots\oplus \cM[(k-1)n]\oplus A[kn]
\end{align*}
which is equivalent to
\begin{align*}
DA[-n]\oplus \cM\oplus\cdots\oplus \cM[(k-1)n]
\end{align*}
On the other hand, by using symmetry structure of $\cM[n\mathbb{Z}]\subset D^b(\modd A)$, we can apply $[-n]$ to $DA\oplus \cM[n]\oplus\cdots\oplus \cM[kn]$, and its endomorphsim algebra is Morita equivalent to $\Sigma^k$. Therefore $\Sigma^k$ can be realized as endomorphism algebra of fundamental domain of higher k-cluster category.
\end{proof}

\begin{corollary}
The quiver of $\Sigma^k$ is simply the Auslander-Reiten quiver of fundamental domain of higher $k$-cluster category. 
\end{corollary}

\subsection{Higher APR-tilting}
O. Iyama and S. Oppermann introduced higher APR tilting in \cite{APR}. Since algebras $\Sigma^k$ are $d$-representation finite, $d$-APR tilts of $\Sigma^k$'s are still $d$-representation finite.  We discuss that how $d$-APR-tilting on $\Sigma^k$ is compatible with the $1$-APR tilting in the derived category. Let $Q$ be a Dynkin quiver, $\Lambda=\mathbb{K}Q$ and $P$ be a simple projective $\Lambda$-module. Then, the fundamental domain of $k$-cluster category of $1$-APR tilted algebra $\Lambda':=\End_{\Lambda}(\Lambda/P\oplus\tau^{-1}P)$  is equivalent to the category generated by $\cF/P\oplus\tau^{-1}P[k]$ where $\add\cF$ is fundamental domain of $k$-cluster category of $\Lambda$ by using derived equivalences of APR tilted algebras. Therefore, we conclude that $\Sigma^k(\Lambda')\cong\End_{D^b(\Lambda)}\left(\Sigma^k(\Lambda)/S\oplus S'\right)$ where $S=\Hom_{\cS^k}(-,P)$ is simple projective $\Sigma^k(\Lambda)$-module and $S'=\Hom_{\cS^k}(-,\tau^{-1}P[k])$.

\section{Uniqueness} \label{section uniqueness}
\subsection{Dynkin Case}
Here we provide proof of Theorems \ref{thm uniqueness gamma dynkin}, \ref{thm uniqueness gamma higher} and \ref{thm uniqueness sigma all}.
We denote the bounded derived category of $\modd\Lambda$ by $D^b(\Lambda)$ in this section.
\begin{proposition}\label{prop uniqueness dynkin 0-1} Let $\cC=\add \cG$ be a convex subcategory of $D^b(\Lambda)$ such that $\cS^0\varsubsetneq \cC \varsubsetneq \cS^1$. Then, $\End_{D^b(\Lambda)}(\cG)$ is not a higher Auslander algebra.
\end{proposition}
\begin{proof}
 Let $\Gamma^1:=\End_{\cS^1}(\cP)$ and $\tilde{\Gamma}^1:=\End_{\cS^1}(\cG)$ where $\add\cP=\cS^1$ and $\add\cG=\cC$. Since $\cC$ is strictly contained in $\cS^1$, there exists an object $X\in\cS^0$ such that $X[1]\notin\cC$. Similarly, $\cS^0$ is strictly contained in $\cC$, there exists $Y\in\cS^0$ such that $Y[1]\in\cC$.

Let $0\rightarrow X\rightarrow I_0\rightarrow I_1\rightarrow 0$ be the injective resolution of $X\in\modd\Lambda$. It is exact, since $\Lambda$ is hereditary. It induces the sequence
\begin{align}
0\rightarrow X\rightarrow I_0\rightarrow I_1\rightarrow X[1]\rightarrow I_0[1]\rightarrow I_1[1]\rightarrow 0
\end{align}
in $\cS^1$. By corollary \ref{cor to lemma hom is exact on triangles}, and proposition \ref{prop gamma k injectivity}
\begin{align*}
0\rightarrow \Hom_{\cS^1}(\cP,X)\rightarrow \Hom_{\cS^1}(\cP,I_0)\rightarrow \Hom_{\cS^1}(\cP,I_1)\rightarrow\\ \Hom_{\cS^1}(\cP,X[1])\rightarrow \Hom_{\cS^1}(\cP,I_0[1])\rightarrow  \Hom_{\cS^1}(\cP,I_1[1])\rightarrow F\rightarrow 0
\end{align*}
is injective $\Gamma^1$-resolution of $\Hom_{\cS^1}(\cP,X)$  where $F$ is injective $\Gamma^1$-module.

In $\tilde{\Gamma}^1$, $\Omega^2\Hom_{\cS^1}(\cP,X)\rightarrow \rad\Hom_{\cS^1}(\cP,X[1])$ exists, because $\Omega^2\Hom_{\cS^1}(\cP,X)$ is proper submodule of $\Hom_{\cS^1}(\cP,X[1])$ in $\Gamma^1$.

We need to show, viewed as $\tilde{\Gamma}^1$-module $\rad\Hom_{\cS^1}(\cP,X[1])$ is injective but not projective. Notice that $\rad\Hom_{\cS^1}(\cP,X[1])$ is injective since 
\begin{align}
\rad\Hom_{\cS^1}(\cP,X[1])\cong D\Hom_{\cS^1}(\tau^{-1}X,\cP)\vert_{\cC}\cong D\Hom_{D^b(\Lambda)}(\tau^{-1}X,\cG)
\end{align}
where $D\Hom_{\cS^1}(\tau^{-1}X,\cP)\vert_{\cC}$ is restriction, is an injective object.

We show that as $\tilde{\Gamma}^1$-module, $\rad\Hom_{\cS^1}(\cP,X[1])$ is not projective.
Let $0\rightarrow \tau X\rightarrow E(X)\rightarrow X\rightarrow 0$ be Auslander-Reiten sequence where $E(X)$ is approximation to $X$ in $\modd\Lambda$. We can choose $X$ such that $E(X)[1]\in\cC$ by convexity of $\cC$. There are two cases we analyze depending on whether $E(X)$ is indecomposable or not. Notice that: $\topp\rad\Hom_{\cS^1}(\cP,X[1])$ are simple functors at $E(X)[1]$. If $E(X)$ is decomposable, $\rad\Hom_{\cS^1}(\cP,X[1])$ cannot be a projective object which violates unique top module assumption.

 Now, we assume that $E(X)$ is indecomposable. Then, $\Hom_{D^b(\Lambda)}(\cG,E(X)[1])\cong D\Hom_{D^b(\Lambda)}(\tau^{-1}E(X),\cG)$ is projective-injective $\tilde{\Gamma}^1$-module. In particular $\tau^{-1}E(X)$ is approximation of $\tau^{-1}X$. Hence $\rad\Hom_{\cS^1}(\cP,X[1])$ is quotient of $\Hom_{D^b(\Lambda)}(\cG,E(X)[1])$ because $\soc\rad\Hom_{\cS^1}(\cP,X[1])$ is contained in the support of $\Hom_{D^b(\Lambda)}(\cG,E(X)[1])$. This shows $\rad\Hom_{\cS^1}(\cP,X[1])$ is not projective. Therefore the injective $\tilde{\Gamma}^1$-resolution
\begin{align}
0\rightarrow \Hom_{\cS^1}(\cP,X)\rightarrow \Hom_{\cS^1}(\cP,I_0)\rightarrow \Hom_{\cS^1}(\cP,I_1)\rightarrow \rad\Hom_{\cS^1}(\cP,X[1])\rightarrow\cdots
\end{align}
shows that $\domdim_{\tilde{\Gamma}^1}\Hom_{\cS^1}(\cP,X)=2$.

On the other hand, if $Y[1]\in\cC$, then
\begin{align}
0\rightarrow \Hom_{\cS^1}(\cP,Y)\rightarrow \Hom_{\cS^1}(\cP,I'_0)\rightarrow \Hom_{\cS^1}(\cP,I'_1)\rightarrow \Hom_{\cS^1}(\cP,Y[1])\rightarrow\cdots
\end{align}
implies that $\domdim_{\tilde{\Gamma}^1}\Hom_{\cS^1}(\cP,Y)\geq 3$, so $\tilde{\Gamma}^1$ cannot be a higher Auslander algebra.
\end{proof}

\begin{proposition}\label{prop uniqueness dynkin i-1-i}Let $\cC=\add \cG$ be a convex subcategory of $D^b(\Lambda)$ such that $\cS^{i-1}\varsubsetneq \cC \varsubsetneq \cS^{i}$. Then, $\End_{D^b(\Lambda)}(\cG)$ is not higher Auslander algebra.
\end{proposition}
\begin{proof}
 Let $\Gamma^i:=\End_{\cS^i}(\cP)$ and $\tilde{\Gamma}^i:=\End_{\cS^i}(\cG)$ where $\add\cP=\cS^i$ and $\add\cG=\cC$. Since $\cC$ is strictly contained in $\cS^{i}$, there exists an object $X\in\cS^0$ such that $X[i]\notin\cC$. Similarly, $\cS^{i-1}$ is strictly contained in $\cC$, there exists $Y\in\cS^0$ such that $Y[i]\in\cC$.\\

Let $0\rightarrow X\rightarrow I_0\rightarrow I_1\rightarrow 0$ be the injective resolution of $X\in\modd\Lambda$. It is exact, since $\Lambda$ is hereditary. It induces the sequence
\begin{align}
0\rightarrow X\rightarrow I_0\rightarrow I_1\rightarrow X[1]\rightarrow I_0[1]\rightarrow I_1[1]\rightarrow \cdots \\
\cdots\rightarrow X[i]\rightarrow I_0[i]\rightarrow I_1[i]\rightarrow 0
\end{align}
in $\cS^i$. Applying functor $\Hom_{\cS^i}(\cP,-)$ gives injective $\Gamma^i$-resolution of $\Hom_{\cS^i}(\cP,X)$. Notice that 
$\domdim_{\Gamma^i}\Hom_{\cS^i}(\cP,X)=3j+\domdim_{\Gamma^{i}}\Hom_{\cS^{i}}(\cP,X[j])$ for $j<i$. Since $X[i-1]\in\cC$, we conclude that
\begin{align}
\domdim_{\tilde{\Gamma}^i}\Hom_{D^b(\Lambda)}(\cG,X)=3(i-1)+\domdim_{\tilde{\Gamma}^i}\Hom_{D^b(\Lambda)}(\cG,X[i-1])
\end{align}
Since $\Hom_{D^b(\Lambda)}(\cG,X[i-1])\cong\Hom_{D^b(\Lambda)}(\cG',X)$ where $\cS^0\varsubsetneq\add\cG'\varsubsetneq\cS^1$ we get 
\begin{align}\label{eq uniqueness dynkin}
\domdim_{\tilde{\Gamma}^i}\Hom_{D^b(\Lambda)}(\cG,X)=3(i-1)+\domdim_{\tilde{\Gamma}^1}\Hom_{D^b(\Lambda)}(\cG',X).
\end{align}
By proposition \ref{prop uniqueness dynkin 0-1}, $\domdim_{\tilde{\Gamma}^1}\Hom_{D^b(\Lambda)}(\cG',X)=2$, hence $\domdim_{\tilde{\Gamma}^i}\Hom_{D^b(\Lambda)}(\cG,X)=3i-1$. On the other hand, for $Y[i]\in\cC$, 
\begin{align}
\domdim_{\tilde{\Gamma}^i}\Hom_{D^b(\Lambda)}(\cG,Y)=3(i-1)+\domdim_{\tilde{\Gamma}^1}\Hom_{D^b(\Lambda)}(\cG',Y)\geq 3i.
\end{align}

\end{proof}
\begin{proposition}\label{prop uniqueness dynkin 0-i} Let $\cC=\add \cG$ be a convex subcategory of $D^b(\Lambda)$ such that $\cS^{0}\varsubsetneq \cC \varsubsetneq \cS^{i}$. Then, $\End_{D^b(\Lambda)}(\cG)$ is not higher Auslander algebra unless $\cC\ncong\cS^j$ for $1\leq j\leq i$.
\end{proposition}
\begin{proof}
Let $X[i]\in\cC$ and $X[j]\notin\cC$ for $j>i$. Let $Y[j]\in\cC$. Then, by propositions \ref{prop uniqueness dynkin 0-1} and \ref{prop uniqueness dynkin i-1-i}, if $j=i+1$, $\tilde{\Gamma}=\End_{D^b(\Lambda)}(\cG)$ is not higher Auslander algebra. Assume $j>i+2$. Now by \eqref{eq uniqueness dynkin}, we get 
$\domdim_{\tilde{\Gamma}}\Hom_{D^b(\Lambda)}(\cG,X)\leq 3i-1$. However $\domdim_{\tilde{\Gamma}}\Hom_{D^b(\Lambda)}(\cG,Y)\geq 3i$. Therefore $\tilde{\Gamma}$ is not a higher Auslander algebra.
\end{proof}

\begin{theorem}\label{thm in uniqueness section dynkin}
Let $\cC=\add \cG$ be a convex subcategory of $D^b(\Lambda)$ such that $\cS^0\varsubsetneq \cC \varsubsetneq \cS^k$.
The followings are equivalent:
\begin{enumerate}[label=\roman*)]
\item $\End_{D^b(\Lambda)}(\cG)$ is a higher Auslander algebra
\item  $\End_{D^b(\Lambda)}(\cG)\cong \Gamma^i$ for some $i$, $1\leq i\leq k-1$
\item $\cC\cong\cS^i$ for some $i$, $1\leq i\leq k-1$.
\end{enumerate}
\end{theorem}

\begin{proof}
Assume $\tilde{\Gamma}=\End_{D^b(\Lambda)}(\cG)$ is a higher Auslander algebra. Therefore, for any  object $X\in\modd\Lambda$, $X[i]\in\cC$ implies that $(\modd\Lambda)[i]\in\cC$ by propositions \ref{prop uniqueness dynkin 0-1}, \ref{prop uniqueness dynkin 0-i} and \ref{prop uniqueness dynkin i-1-i}. Therefore $\cC
\cong\cS^i$ for some $i$ if and only if $\tilde{\Gamma}\cong\Gamma^i$ if and only if $\cC\cong\cS^i$ for $1\leq i\leq k$.
\end{proof}

\begin{corollary} Let $\cS^i\varsubsetneq\cC\varsubsetneq\cS^j$. Then, 
\begin{gather*}
\domdim\End_{D^b(\Lambda)}(\cG)\leq 3i+2\leq\gldim\End_{D^b(\Lambda)}(\cG)\leq 3j+2.
\end{gather*}
\end{corollary}
\begin{proof}
In the proof of proposition \ref{prop uniqueness dynkin i-1-i}, the inequality \eqref{eq uniqueness dynkin} gives the upper bound for dominant dimension. Since there exists $X[j]\in\cC\cap\cS^j$, and $\gldim\Gamma^j=3j+2$ by Theorem \ref{THM dynkin}, claim follows.
\end{proof}

We restate and prove the part one of Theorem \ref{thm uniqueness sigma all}

\begin{theorem}
Let $\cC=\add\cG$ be a convex subcategory of $D^b(\modd\Lambda)$ such that
\begin{align*}
\cS^0\varsubsetneq \cC \varsubsetneq \cS^k.
\end{align*}
Then $\End_{D^b(\Lambda)}(\cG)$ is an $n$-representation finite algebra if and only if $\End_{D^b(\Lambda)}(\cG)\cong \Sigma^i$ for $1\leq i\leq k-1$ if and only if $\cC\cong\cS^i\oplus\Lambda[i+1] $ for $1\leq i\leq k-1$.
\end{theorem}
\begin{proof}
Let $\tilde{\Sigma}:=\End_{D^b(\modd\Lambda)}(\cG)$. Then, $\tilde{\Sigma}$ is $n$-representation finite, implies that there exists $n$-cluster tilting object $N$. By Theorem \ref{thm uniqueness gamma dynkin}, $\End_{D^b(\modd\Lambda)}(N)\cong\Gamma^i$ for some $i$. Therefore by Theorem \ref{thm sigma is d rep DYNKIN}, $\tilde{\Sigma}\cong\Sigma^i$ for some $i$. 
\end{proof}

\subsection{$n$-Representation Finite Case}

\begin{proposition}\label{prop uniqueness higher 0-1} Let $\cC=\add \cG$ be a convex subcategory of $\cM[n\mathbb{Z}]$ such that $\cM^0\varsubsetneq \cC \varsubsetneq \cM^1$. Then, $\End_{D^b(\Lambda)}(\cG)$ is not a higher Auslander algebra.
\end{proposition}
\begin{proof}
 Let $\Gamma^1:=\End_{\cM^1}(\cP)$ and $\tilde{\Gamma}^1:=\End_{\cM^1}(\cG)$ where $\add\cP=\cM^1$ and $\add\cG=\cC$. Since $\cC$ is strictly contained in $\cM^1$, there exists an object $X\in\cM^0$ such that $X[n]\notin\cC$. Similarly, $\cM^0$ is strictly contained in $\cC$, there exists $Y\in\cM^0$ such that $Y[n]\in\cC$.

Since $A$ is $n$-representation finite and $n$-hereditary algebra, the injective resolution
 $0\rightarrow X\rightarrow I_0\rightarrow I_1\rightarrow \cdots$  is exact and induces the sequence 
\begin{align*}
0\rightarrow X\rightarrow I_0\rightarrow I_1 \rightarrow\cdots\rightarrow I_{n}[n]\rightarrow 0
\end{align*}
in $\cM^1$. Therefore, the sequence
\begin{align*}
0\rightarrow\Hom_{\cM^1}(\cP,X)\rightarrow \Hom_{\cM^1}(\cP,I_0)\rightarrow\Hom_{\cM^1}(\cP, I_1) \rightarrow\Hom_{\cM^1}(\cP,I_2)\rightarrow\cdots\\
\cdots\rightarrow\Hom_{\cM^1}(\cP,I_{n-1}[n])\rightarrow\Hom_{\cM^1}(\cP,I_{n}[n])\rightarrow F\rightarrow 0
\end{align*}
is injective $\Gamma^1$-resolution of $\Hom_{\cM^1}(\cP,X)$  where $F$ is injective $\Gamma^1$-module.

In $\tilde{\Gamma}^1$, $\Omega^{n+1}\Hom_{\cM^1}(\cP,X)\rightarrow \rad\Hom_{\cM^1}(\cP,X[n])$ exists, because $\Omega^{n+1}\Hom_{\cM^1}(\cP,X)$ is proper submodule of $\Hom_{\cM^1}(\cP,X[n])$ in $\Gamma^1$.

We need to show, viewed as $\tilde{\Gamma}^1$-module $\rad\Hom_{\cM^1}(\cP,X[n])$ is injective but not projective. Notice that $\rad\Hom_{\cM^1}(\cP,X[n])$ is injective since 
\begin{align}
\rad\Hom_{\cM^1}(\cP,X[n])\cong D\Hom_{\cM^1}(\tau^{-1}_nX,\cP)\vert_{\cC}\cong D\Hom_{D^b(A)}(\tau^{-1}_nX,\cG)
\end{align}
where $D\Hom_{\cS^1}(\tau^{-1}_nX,\cP)\vert_{\cC}$ is restriction, is an injective object.

We show that as $\tilde{\Gamma}^1$-module, $\rad\Hom_{\cM^1}(\cP,X[n])$ is not projective.
Let $0\rightarrow \tau X\rightarrow E_n(X)\rightarrow\cdots\rightarrow E_1(X)\rightarrow X\rightarrow 0$ be a higher Auslander-Reiten sequence where $E(X)=E_0(X)$ is approximation to $X$ in $\cM$. We can choose $X$ such that $E(X)[n]\in\cC$ by convexity of $\cC$. There are two cases we analyze depending on whether $E(X)$ is indecomposable or not. Notice that: $\topp\rad\Hom_{\cM^1}(\cP,X[n])$ are simple functors at $E(X)[1]$. If $E(X)$ is decomposable, $\rad\Hom_{\cM^1}(\cP,X[n])$ cannot be a projective object which violates unique top module assumption.

 Now, we assume that $E(X)$ is indecomposable. Then, $\Hom_{D^b(A)}(\cG,E(X)[n])\cong D\Hom_{D^b(A)}(\tau^{-1}_nE(X),\cG)$ is projective-injective $\tilde{\Gamma}^1$-module. In particular $\tau^{-1}_nE(X)$ is approximation of $\tau^{-1}_nX$. Hence $\rad\Hom_{\cM^1}(\cP,X[n])$ is quotient of $\Hom_{D^b(A)}(\cG,E(X)[n])$ because $\soc\rad\Hom_{\cM^1}(\cP,X[n])$ is contained in the support of $\Hom_{D^b(A)}(\cG,E(X)[n])$. This shows $\rad\Hom_{\cM^1}(\cP,X[n])$ is not projective. Therefore the injective $\tilde{\Gamma}^1$-resolution
\begin{align}\label{eq domdim is n+1 section uniqueness higher case}
0\rightarrow \Hom_{\cM^1}(\cP,X)\rightarrow \Hom_{\cM^1}(\cP,I_0)\rightarrow \Hom_{\cM^1}(\cP,I_1)\rightarrow\ldots\rightarrow \rad\Hom_{\cM^1}(\cP,X[n])\rightarrow\cdots
\end{align}
shows that $\domdim_{\tilde{\Gamma}^1}\Hom_{\cM^1}(\cP,X)=n+1$.

On the other hand, if $Y[n]\in\cC$, then
\begin{align}
0\rightarrow \Hom_{\cM^1}(\cP,Y)\rightarrow \Hom_{\cM^1}(\cP,I'_0)\rightarrow \Hom_{\cM^1}(\cP,I'_1)\rightarrow\cdots\rightarrow \Hom_{\cM^1}(\cP,Y[n])\rightarrow\cdots
\end{align}
implies that $\domdim_{\tilde{\Gamma}^1}\Hom_{\cM^1}(\cP,Y)\geq n+2$, so $\tilde{\Gamma}^1$ cannot be a higher Auslander algebra.
\end{proof}

\begin{proposition}\label{prop uniqueness higher i-1-i}Let $\cC=\add \cG$ be a convex subcategory of $D^b(A)$ such that $\cM^{i-1}\varsubsetneq \cC \varsubsetneq \cM^{i}$. Then, $\End_{D^b(A)}(\cG)$ is not higher Auslander algebra.
\end{proposition}
\begin{proof}
 Let $\Gamma^i:=\End_{\cM^i}(\cP)$ and $\tilde{\Gamma}^i:=\End_{\cM^i}(\cG)$ where $\add\cP=\cM^i$ and $\add\cG=\cC$. Since $\cC$ is strictly contained in $\cM^{i}$, there exists an object $X\in\cM^0$ such that $X[in]\notin\cC$. Similarly, $\cM^{i-1}$ is strictly contained in $\cC$, there exists $Y\in\cS^0$ such that $Y[in]\in\cC$.\\

Let $0\rightarrow X\rightarrow I_0\rightarrow I_1\rightarrow \cdots$ be the injective resolution of $X\in\cM$. It is exact, since $A$ is $n$-representation finite. It induces the sequence
\begin{align}
0\rightarrow X\rightarrow I_0\rightarrow I_1\rightarrow \cdots\rightarrow I_n\rightarrow X[n]\rightarrow I_0[n]\rightarrow \cdots I_{n-1}[in]\rightarrow I_n[in]\rightarrow 0
\end{align}
in $\cM^i$. Applying functor $\Hom_{\cM^i}(\cP,-)$ gives injective $\Gamma^i$-resolution of $\Hom_{\cM^i}(\cP,X)$. Notice that 
$\domdim_{\Gamma^i}\Hom_{\cM^i}(\cP,X)=(n+2)j+\domdim_{\Gamma^{i}}\Hom_{\cM^{i}}(\cP,X[jn])$ for $j<i$. Since $X[(i-1)n]\in\cC$, we conclude that
\begin{align}
\domdim_{\tilde{\Gamma}^i}\Hom_{D^b(A)}(\cG,X)=(n+2)(i-1)+\domdim_{\tilde{\Gamma}^i}\Hom_{D^b(A)}(\cG,X[(i-1)n])
\end{align}
Since $\Hom_{D^b(A)}(\cG,X[(i-1)n])\cong\Hom_{D^b(A)}(\cG',X)$ where $\cM^0\varsubsetneq\add\cG'\varsubsetneq\cM^1$ we get 
\begin{align}\label{eq eq domdim}
\domdim_{\tilde{\Gamma}^i}\Hom_{D^b(A)}(\cG,X)=(n+2)(i-1)+\domdim_{\tilde{\Gamma}^1}\Hom_{D^b(A)}(\cG',X).
\end{align}
By \eqref{eq domdim is n+1 section uniqueness higher case} in the proof of proposition \ref{prop uniqueness dynkin 0-1}, $\domdim_{\tilde{\Gamma}^1}\Hom_{D^b(A)}(\cG',X)=n+1$, hence $\domdim_{\tilde{\Gamma}^i}\Hom_{D^b(A)}(\cG,X)(n+2)i-1$. On the other hand, for $Y[in]\in\cC$, 
\begin{align}
\domdim_{\tilde{\Gamma}^i}\Hom_{D^b(A)}(\cG,Y)=(n+2)(i-1)+\domdim_{\tilde{\Gamma}^1}\Hom_{D^b(A)}(\cG',Y)\geq (n+2)i.
\end{align}

\end{proof}

\begin{proposition}\label{prop uniqueness higher 0-i} Let $\cC=\add \cG$ be a convex subcategory of $D^b(A)$ such that $\cM^{0}\varsubsetneq \cC \varsubsetneq \cM^{i}$. Then, $\End_{D^b(A)}(\cG)$ is not higher Auslander algebra unless $\cC\ncong\cM^j$ for $1\leq j\leq i$.
\end{proposition}
\begin{proof}
Let $X[in]\in\cC$ and $X[jn]\notin\cC$ for $j>i$. Let $Y[jn]\in\cC$. Then, by propositions \ref{prop uniqueness higher 0-1} and \ref{prop uniqueness higher i-1-i}, if $j=i+1$, $\tilde{\Gamma}=\End_{D^b(A)}(\cG)$ is not higher Auslander algebra. Assume $j>i+1$. Now by \eqref{eq eq domdim}, we get 
$\domdim_{\tilde{\Gamma}}\Hom_{D^b(A)}(\cG,X)\leq (n+2)i-1$. However $\domdim_{\tilde{\Gamma}}\Hom_{D^b(A)}(\cG,Y)\geq (n+2)i$. Therefore $\tilde{\Gamma}$ is not a higher Auslander algebra.
\end{proof}

\begin{theorem}
Let $\cC=\add \cG$ be a convex subcategory of $\cM[n\mathbb{Z}]$ such that $\cS^0\varsubsetneq \cC \varsubsetneq \cS^k$.
The followings are equivalent:
\begin{enumerate}[label=\roman*)]
\item $\End_{D^b(A)}(\cG)$ is a higher Auslander algebra
\item  $\End_{D^b(A)}(\cG)\cong \Gamma^i$ for $1\leq i\leq k-1$
\item $\cC\cong\cS^i$ for $1\leq i\leq k-1$.
\end{enumerate}
\end{theorem}

\begin{proof}
Assume $\tilde{\Gamma}=\End_{D^b(A)}(\cG)$ is a higher Auslander algebra. Therefore, for any  object $X\in\cM$, $X[in]\in\cC$ implies that $\cM[in]\in\cC$ by propositions \ref{prop uniqueness higher 0-1}, \ref{prop uniqueness higher 0-i} and \ref{prop uniqueness higher i-1-i}. Therefore $\cC
\cong\cM^i$ for some $i$ if and only if $\tilde{\Gamma}\cong\Gamma^i$ if and only if $\cC\cong\cM^i$ for $1\leq i\leq k$.
\end{proof}

We prove part 2 of Theorem \ref{thm uniqueness sigma all}

\begin{theorem}
Let $\cC=\add\cG$ be a convex subcategory of $\cM[n\mathbb{Z}]$ such that
\begin{align*}
\cM^0\varsubsetneq \cC \varsubsetneq \cM^k.
\end{align*}
Then $\End_{D^b(A)}(\cC)$ is an $n$-representation finite algebra if and only if $\End_{D^b(A)}(\cC)\cong \Sigma^i$ for $1\leq i\leq k-1$ if and only if $\cC\cong\cM^i\oplus A[(i+1)n]$ for $1\leq i\leq k-1$.
\end{theorem}

\begin{proof}
Let $\tilde{\Sigma}:=\End_{D^b(A)}(\cG)$. Then, $\tilde{\Sigma}$ is $n$-representation finite, implies that there exists $n$-cluster tilting object $N$. By Theorem \ref{thm uniqueness gamma higher}, $\End_{D^b(A)}(N)\cong\Gamma^i$ for some $i$. Therefore by Theorem \ref{thm sigma is d rep HIGHER}, $\tilde{\Sigma}\cong\Sigma^i$ for some $i$. 
\end{proof}

\subsection{Applications to Higher Nakayama Algebras}\label{applications to nakayama}
As an application of Theorem \ref{thm uniqueness sigma all}, we provide a class of higher Nakayama algebras that are $d$-representation finite. Similarly, by Theorems \ref{thm uniqueness gamma dynkin} and \ref{thm uniqueness gamma higher}, we obtain a class of higher Nakayama algebras that are higher Auslander algebras. The idea we implement is the following: we express a certain class of higher Nakayama algebras as endomorphism algebras of suitable convex subcategories.

We briefly recall some properties of higher Nakayama algebras and for details refer to \cite{julianNakayama}.
Let $\ell=(\ell_0,\ldots,\ell_{d-1})$ be a Kupisch series, $\ell_0=1$ and for all $i\geq 2$ there are inequalities $2\leq \ell_i\leq \ell_{i-1}+1$. Then, $d$-tuples $(i_1,\ldots,i_d)$ satisfying $i_d-i_1+1\leq \ell_{i_{d}}$ forms vertices of the higher Nakayama algebra $B$ where projective object at the vertex $(i_1,\ldots,i_d)$ has socle $(i_d+1-\ell_{i_d},i_2,\ldots,i_{d-1})$. Iyama's higher $\mathbb{A}$-type algebras are examples of higher Nakayama algebras given by Kupisch series $(1,2,3,\ldots,n)$ for any dimension $d$. We recall that $d$-cluster tilting object of an algebra $A$ is called $d\mathbb{Z}$-cluster tilting object if $d$ divides global dimension of $A$. An important feature of $d$-dimensional higher Nakayama is that they always have $d\mathbb{Z}$-cluster tilting objects.\\

For this subsection we fix the notation. We always assume that Kupisch series $\ell$ is given by
\begin{align}\label{eq definition of kupsich series ell}
\ell=(1,2,\ldots,m-1,\underbrace{m-1,\ldots,m-1}_\textrm{a-many})
\end{align}
where $a\geq 0$. Iyama's $d$-dimensional $\mathbb{A}$-type algebra arising from $\mathbb{A}_m$ is denoted by $\mathbb{A}^d_m$. Since $\mathbb{A}^d_m$ is $d$-representation finite and higher Auslander algebra, we can apply Theorems \ref{THM higher hereditary} and \ref{THM sigma k algebra}.

First, we need the following key observation. 
\begin{proposition}\label{prop B is nakayama and endomorphism} Let $\ell=(1,2,\ldots,m-1,m-1,\ldots,m-1)$ be a Kupisch series. Let $B:=\mathbb{A}^d_{\underline{\ell}}$ be the $d$-dimensional linear higher Nakayama algebra associated to $\ell$. Then, there exists a convex subcategory $\cG$ of $D^b(\mathbb{A}^{d-1}_{m})$ such that $B\cong\End_{D^b(\mathbb{A}^{d-1}_{m})}\cG$.
\end{proposition} 

\begin{proof}
First, we analyze the case $d=2$. By definition of $B$, simple $B$-modules are tuples $(i,j)$ such that 
\begin{align*}
0\leq i\leq j\leq m-2+a,\,\,\,\, j-i\leq m-2
\end{align*}
where $a+1$ is the number of appearances of $m-1$ in $\ell$. By definition, projective module at simple $(i,j)$ has socle given by $(j-m+2,i)$. Now, we can identify each indecomposable projective module of $\mathbb{A}_m$ by $(0,s-1)$ where $s=\ell(P)$. That is, simple projective module corresponds to $(0,0)$, rank $2$ projective module corresponds to $(0,1)$ etc. It is clear that $\tau^{-1}$ orbits of these modules  forms the derived category, since $\tau^{-1}(i,j)\cong(i+1,j+1)$ and $\tau(i,j)\cong(i-1,j-1)$ by Auslander-Reiten quiver of $D^b(\mathbb{A}_m)$. In particular, by $\ell$, this is a convex region. On the other hand, $(i,j)[1]\cong(j-m+1,i-1)$ by direct calculation. Therefore, $\tau^{-1}(i,j)[1]\cong(j-m+2,i)$. Hence if $X\in \cG/\modd\mathbb{A}_m$ is identified with $(i,j)$, then socle of $\Hom_{D^b(\mathbb{A}_m)}(\cG,X)$ is $(j-m+2,i)$ by Serre duality, i.e., $\soc \Hom_{D^b(\mathbb{A}_m)}(\cG,X)\cong \topp D\Hom_{D^b(\mathbb{A}_m)}(\tau^{-}X[1],\cG)$. For $X\in \modd\mathbb{A}_m$, then socle is $(0,i)$ which fits the description of $B$. Hence, the category of projective $B$-modules is equivalent to the category of projective $\End_{D^b(\mathbb{A}_m)}\cG$-modules. \\

We analyze the case $d\geq 3$. Simple $B$-modules are $d$-tuples $(i_1,\ldots,i_d)$ such that 
\begin{align*}
0\leq i_1\leq i_2\leq\cdots\leq i_d,\,\,\,\,\, i_d-i_1\leq m-2.
\end{align*}
Projective $B$-module at $(i_1,\ldots,i_d)$ has socle $(i_d-m+2,\ldots,i_1)$ by definition of higher Nakayama algebras. We can identify indecomposable modules of $d$-cluster tilting object of $\mathbb{A}^{d-1}_m$ by all tuples $(i_1,\ldots,i_d)$ such that $0\leq i_1\leq i_2\leq\cdots\leq i_d\leq m-2+a$. Hence $(d+2)$-angulated category by increasing the entries, that is $\tau_{d-1}(i_1,\ldots,i_d)=(i_1+1,\ldots,i_d+1)$. Notice that this is convex region. On the other hand $(i_1,\ldots,i_d)[d-1]=(i_d-m+1,\ldots,i_{d-1})$. If we identify any $X\in\modd\mathbb{A}^{d-1}_m$ by $(i_1,\ldots,i_d)$, then $\soc\Hom_{D^b(\mathbb{A}^{d-1}_{m})}(\cG,X)$ is given by functor at $\tau^{-1}_{d-1}[d-1]$ by Serre duality. Hence category of projective $B$-modules is equivalent to the category of $\End_{D^b(\mathbb{A}^{d-1}_m)}\cG$-modules. Therefore $B\cong \End_{D^b(\mathbb{A}^{d-1}_{m})}\cG$.
\end{proof}

By proposition \ref{prop B is nakayama and endomorphism}, we give another way to construct higher Nakayama algebras given  by Kupisch series $\ell=(1,2,\ldots,m-1,\ldots,m-1)$. As an application we show that:
\begin{proposition}\label{prop B is nakayama iff sigma-gamma}
Let $B:=\mathbb{A}^{d}_{\underline{\ell}}$ be the $d$-dimensional linear higher Nakayama algebra given by $\ell$ \eqref{eq definition of kupsich series ell}. Let $\cM$ be $(d-1)$-cluster tilting object of $\mathbb{A}^{d-1}_{m}$. Then,
\begin{enumerate}[label=\arabic*)]
\item $B$ is $s$-representation finite iff $B\cong\Sigma^k$ for some $k$ where $\Sigma^k:=\End_{D^b(\mathbb{A}^{d-1}_{m})}(\cM^{k-1}\oplus\mathbb{A}^{d-1}_{m}[k(d-1)])$.
\item $B$ is a higher Auslander algebra iff $B\cong\Gamma^k$ for some $k$ where $\Gamma^k:=\End_{D^b(\mathbb{A}^{d-1}_{m})}(\cM^k)$.
\end{enumerate}
\end{proposition}
\begin{proof}
By proposition \ref{prop B is nakayama and endomorphism}, $B$ is endomorphism algebra of convex subcategory of $D^b(\mathbb{A}^{d-1}_{m})$. By Theorems  \ref{thm uniqueness gamma dynkin}, \ref{thm uniqueness gamma higher} and \ref{thm uniqueness sigma all} the claim follows.
\end{proof}

Now, we can describe numerical values of $a$, $k$ and $s$.

\begin{proposition}\label{prop B is s-rep finite numerical}
Let $\ell=(1,2,\ldots,m-1,m-1,\ldots,m-1)$ \eqref{eq definition of kupsich series ell}. Let $B=\mathbb{A}^d_{\underline{\ell}}$. B is $s$-representation finite if $k\equiv 1\mod d$, $a=1+(k-1)\frac{m+d-1}{d}$ and $s=(d+1)k+d-1$. 
\end{proposition}
\begin{proof}
$B$ has $d\mathbb{Z}$-cluster tilting object. Therefore
\begin{align}
d\mid \gldim\Sigma^k\implies d\mid k(d+1)+d-1\implies k\equiv 1\mod d.
\end{align}
We will compute ranks of $B$ and $\Sigma^k$. First of all, the rank of $\mathbb{A}^d_{m}$ is given by $\binom{m+d-1}{d}$. Therefore the rank of $\Sigma^k$ is 
\begin{align*}
k\binom{m+d-1}{d}+\binom{m+d-2}{d-1}.
\end{align*}
On the other hand, if $a=0$, then the rank of $B$ is $\binom{m+d-1}{d}$. For each increment of $a$ by $1$, we add $\binom{m+d-2}{d-1}$ terms. So the rank of $B$ is 
\begin{align*}
\binom{m+d-1}{d}+a\binom{m+d-2}{d-1}.
\end{align*}
Hence by proposition \ref{prop B is nakayama and endomorphism} we need to solve the equation:
\begin{align*}
k\binom{m+d-1}{d}+\binom{m+d-2}{d-1}=\binom{m+d-1}{d}+a\binom{m+d-2}{d-1}.
\end{align*}
We get 
\begin{gather*}
(k-1)\binom{m+d-1}{d}=(a-1)\binom{m+d-2}{d-1}\implies\\
1+(k-1)\frac{m+d-1}{d}=a
\end{gather*}

\end{proof}

These numerical values gives a class of higher Nakayama algebras which are $s$-representation finite.

Now, we can describe the numerical values of $a$, $k$ and $s$ which makes $B$ a higher Auslander algebra.

\begin{proposition}\label{prop B is higher auslander numerical}
Let $\ell=(1,2,\ldots,m-1,m-1,\ldots,m-1)$ \eqref{eq definition of kupsich series ell}. Let $B=\mathbb{A}^d_{\underline{\ell}}$. B is a higher Auslander algebra if $k\equiv 0\mod d$ and $a=k\frac{m+d-1}{d}$. 
\end{proposition}
\begin{proof}
Since $B\cong\Gamma^k$, and has $d\mathbb{Z}$-cluster tilting object, we get $d$ divides global dimension of $\Gamma^k$ which is $(d+1)k+d$. Hence $k\equiv 0\mod d$. Similar to the proof of the previous proposition, we need to express ranks of $B$ and $\Gamma^k$. Since the rank of  $\Gamma^k$ is $(k+1)\binom{m+d-1}{d}$, we get
\begin{gather*}
(k+1)\binom{m+d-1}{d}=\binom{m+d-1}{d}+a\binom{m+d-2}{d-1}\implies\\
k\frac{m+d-1}{d}=a.
\end{gather*}
\end{proof}

\begin{corollary}
Let $\cM$ be a $d$-cluster tilting object of algebra $\mathbb{A}^{d}_m$. Let $k\equiv 1\mod d+1$. Then,
\begin{align*}
\Sigma^k:=\End_{D^b(\mathbb{A}^{d}_m)}\left(\cM^{k-1}\oplus\mathbb{A}^{d}_m[kd]\right)
\end{align*}
is a $(d+2)k+d$-representation finite algebra having $(d+1)\mathbb{Z}$-cluster tilting object. 
\end{corollary}

\begin{proof}
Since $k\equiv 1\mod d+1$, by propositions \ref{prop B is nakayama iff sigma-gamma}, \ref{prop B is s-rep finite numerical} $\Sigma^k$ is a $(d+1)$-dimensional Nakayama algebra. Hence it has $(d+1)\mathbb{Z}$-cluster tilting object.
\end{proof}

\begin{corollary}
Let $\cM$ be a $d$-cluster tilting object of algebra $\mathbb{A}^{d}_m$. Let $k\equiv 0 \mod (d+1)$. Then,
\begin{align*}
\Gamma^k:=\End_{D^b(\mathbb{A}^{d}_m)}\left(\cM^{k}\right)
\end{align*}
is a higher Auslander algebra of global dimension $(d+2)k+d+1$ having $(d+1)\mathbb{Z}$-cluster tilting object. 
\end{corollary}

\begin{proof}
Since $k\equiv 0\mod d+1$, by propositions \ref{prop B is nakayama iff sigma-gamma}, \ref{prop B is s-rep finite numerical} $\Sigma^k$ is a $(d+1)$-dimensional Nakayama algebra. Hence it has $(d+1)\mathbb{Z}$-cluster tilting object.
\end{proof}

\section{Final Remarks \& Examples}\label{section examples}

Here we collect some remarks and examples.

\begin{remark} If $\Lambda$ is representation finite algebra but not Dynkin type, the Theorem \ref{THM dynkin} is not true anymore. Because in the derived category, maps between $X,Y\in\modd\Lambda$ can factor thorough $G[-1]$ or $G[1]$ where $\add G=\modd\Lambda$. Similarly, Theorem \ref{THM higher hereditary} fails if $A$ is not $n$-representation finite $n$-hereditary algebra.
\end{remark}

\begin{remark} In general, gluing two higher Auslander algebras is not higher Auslander. Surprisingly, the algebras $\Gamma^k$ can be realized as a result of gluing. We show this on an example.

\begin{example}\label{example gluing aus3} Let $\Gamma_1,\Gamma_2$ be Auslander algebras of straightly oriented $\mathbb{A}_3$ quivers. Then $\Omega^2$ induces a bijection between  injective non-projectives and projective non-injectives of $\modd\Gamma_i$. We can construct $\Gamma^1$ as extending projective but non-injectives of $\Gamma_1$ by injective but non-projectives of $\Gamma_2$.
Quivers of $\Gamma_1$, $\Gamma_2$ are
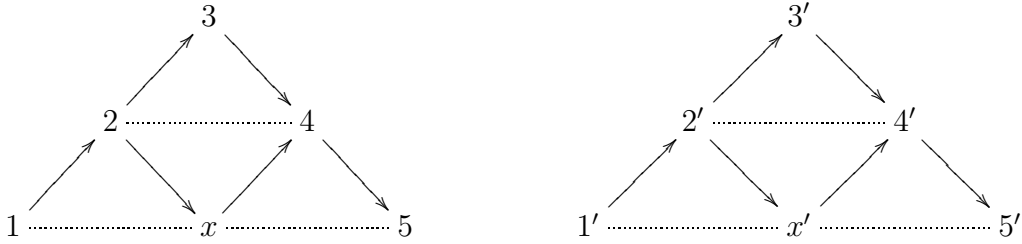
\begin{align*}
\xymatrix{
&&3\ar[rd]&&   &&&& 3'\ar[rd]\\
&2\ar[ru]\ar[rd]\ar@{..}[rr]&&4\ar[rd]& &&&2'\ar[ru]\ar@{..}[rr]\ar[rd]&&4'\ar[rd]\\
1\ar[ru]\ar@{..}[rr]&&x\ar@{..}[rr]\ar[ru]&&5&&1'\ar@{..}[rr]\ar[ru]&&x'\ar@{..}[rr]\ar[ru]&&5'}
\end{align*}
where dotted lines denotes relations.

We can create the algebra $\Gamma'$ by defining nontrivial extensions
\begin{align}\label{eq example gluing A3}
 \begin{vmatrix}
1'
\end{vmatrix}\rightarrow E_1\rightarrow\begin{vmatrix}
x\\4
\end{vmatrix},\quad
\begin{vmatrix}
1'\\2'
\end{vmatrix}\rightarrow E_2\rightarrow\begin{vmatrix}
4\\5
\end{vmatrix},\quad
\begin{vmatrix}
2'\\x'
\end{vmatrix}\rightarrow E_3\rightarrow\begin{vmatrix}
5
\end{vmatrix}
\end{align}
So $\Gamma'$ is given by quiver
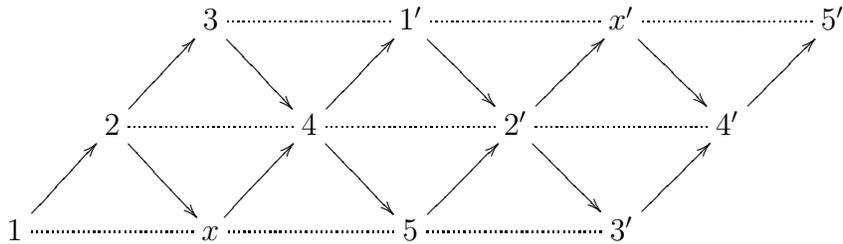
\begin{align*}
\xymatrix{
&& 3\ar[rd] \ar@{..}[rr] && 1' \ar@{..}[rr]\ar[rd]&&x' \ar@{..}[rr]\ar[rd]&&5'\\
&2\ar[ru] \ar@{..}[rr]\ar[rd]&&4 \ar@{..}[rr]\ar[rd]\ar[ru]&&2' \ar@{..}[rr]\ar[rd]\ar[ru]&&4'\ar[ru]\\
1 \ar@{..}[rr]\ar[ru]&&x \ar@{..}[rr]\ar[ru]&&5 \ar@{..}[rr]\ar[ru]&&3'\ar[ru]}
\end{align*}

 The resulting algebra is indeed $\Gamma^1$ and it is higher Auslander algebra of global dimension $5$. Consider the injective resolution of $I_2\in\modd\Gamma^1$
 \begin{align*}
 \xymatrixcolsep{4mm}
\xymatrix{&P_{3'}\ar[rd]\ar[rr]&&P_{1'}\ar[rd]\ar[rr]&&P_{4}\ar[rd]\ar[rr]&&P_{3}\ar[rd]\ar[rr]&&P_{1}\ar[rr]&&I_2\\
P_{4'}\ar[ru]&&S_{3'}\ar[ru]&&I_{2'}\ar[ru]&&\Omega^2(I_2)\ar[ru]&&S_{3}\ar[ru]&&}
 \end{align*}
 where $\Omega^2(I_2)=\begin{vmatrix}
 S_4\\S_5
 \end{vmatrix}$ which has the same structure of $P_4$ of $\Gamma_1$ algebra. Notice that $P_4$ is isomorphic to $E_2$ \ref{eq example gluing A3}, and $\Omega^3(I_{2})\cong I_{2'}$.
\end{example}

\end{remark}

\begin{example}
In \cite{vaso2019n}, $n$-representation finite Nakayama algebras were classified. We want to show how Theorem \ref{THM sigma k algebra} applies. Let $L$  be homogeneous linear Nakayama algebra of rank $m$ and length of projective-injective modules be $\ell>2$. Then the rank of $d$-cluster tilting object is $m+\ell-1$. The rank of $\Sigma^k$ is $k(m+\ell-1)+m$ and its global dimension is $(d+2)k+d$. $L$ is $d$-representation finite if $d\ell=2(m-1)$. If we apply this to $\Sigma^k$, we get 
\begin{align*}
2\cfrac{k(m+\ell-1)+m-1}{\ell}=2\cfrac{(k+1)(m-1)+k\ell}{\ell}=d(k+1)+2k=(d+2)k+d
\end{align*}
Hence, $\Sigma^k$ is $(d+2)k+d$-representation finite algebra.\\
By propositions \ref{prop B is nakayama iff sigma-gamma} and \ref{prop B is s-rep finite numerical}, we give a class of $d$-dimensional Nakayama algebras which are $s$-representation finite. Now, we show that the case $d=1$ recovers the result of \cite{vaso2019n}. Let $d=1$ and $\boldsymbol{\ell}=(1,2,\ldots,\ell,\underbrace{\ell,\ldots,\ell}_\textrm{a-many})$.
Then, we get $s=2k$, $a=1+(k-1)(\ell+1)$ by proposition \ref{prop B is s-rep finite numerical}. This implies $s(\ell+1)=2(a+\ell)$. If we change the convention to count indices of classical Nakayama algebras starting from zero instead of one, this means $\ell'=\ell+1$, and in particular $s\ell'=2(a+\ell'-1)$, where $a+\ell'$ is the length of Kupisch series $\boldsymbol{\ell}$. This suggests the following question: Is it true that a higher Nakayama algebra is $s$-representation finite if and only if it is given by a Kupisch series \eqref{eq definition of kupsich series ell} satisfying the numerical conditions in Proposition~\ref{prop B is s-rep finite numerical}?
 \end{example}
 \begin{example} Let $d=2$, $m=4$. Then, the higher Nakayama algebra $B$ given by Kupisch series $\ell=(1,2,3,3)$ is given by quiver

\begin{align*}
\xymatrix{
&& 02\ar[rd]  && 13\ar[rd] &&&&\\
&01\ar[ru] \ar[rd]&&12 \ar[rd]\ar[ru]&&23 \ar[rd]&&\\
00 \ar[ru]&&11 \ar[ru]&&22 \ar[ru]&&33}
\end{align*}
and projective object at $(i,j)$ is the interval module $M[(i,j),(x,i)]$ where $x=0$ if $\leq 2$ and $x=1$ if $j=3$. We also include Auslander-Reiten quiver of $\modd\mathbb{A}_3\oplus \mathbb{A}_3[1]$:
\begin{align*}
\xymatrix{
&& P_1\ar[rd]  && P_3[1]\ar[rd] &&&&\\
&P_2\ar[ru] \ar[rd]&&I_2 \ar[rd]\ar[ru]&&P_2[1] \ar[rd]&&\\
P_3 \ar[ru]&&S_2 \ar[ru]&&I_1\ar[ru]&&P_1[1]}
\end{align*}
where $P_3$ is simple projective, $P_1$ is projective-injective of $\modd\mathbb{A}_3$ of quiver $\mathbb{A}_3$:
\begin{align*}
\xymatrix{1\ar[r]&2\ar[r]&3}.
\end{align*}
By proposition \ref{prop B is nakayama iff sigma-gamma}, $B\cong\Sigma^1$ which is the endomorphism algebra of the fundamental domain of cluster category of $\mathbb{A}_3$. In particular, $\Sigma^1\cong\End_{\Gamma'}{Q}$ where $Q$ is projective-injective module of $\Gamma'$ in the example \ref{example gluing aus3}.
\end{example}

\subsection{Example $D_4$}

We consider $D_4$ quiver 
\begin{align*}
\xymatrix{
    1 \ar[r] \ar[d]\ar[rd] & 2  \\
    3        & 4 }
\end{align*}

Auslander Reiten quiver of the category $\cS^1$ is

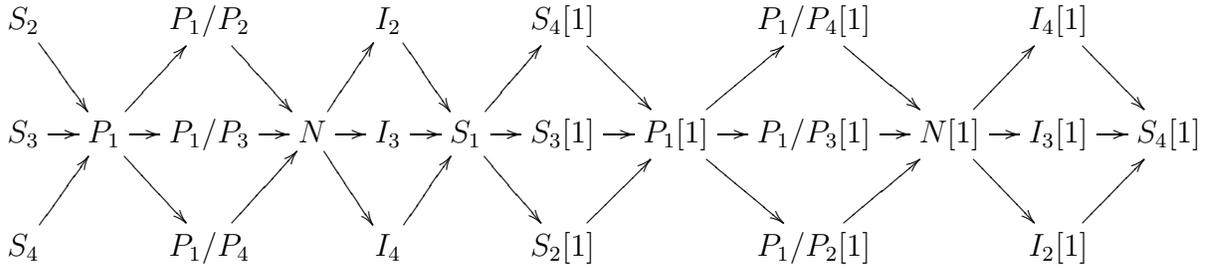
\begin{align*}
\xymatrixcolsep{4mm}
\xymatrix{
   S_2\ar[rd] && P_1/P_2\ar[rd] && I_2\ar[rd] && S_4[1]\ar[rd] && P_1/P_4[1]\ar[rd]&&I_4[1]\ar[rd]&\\
   S_3\ar[r]& P_1 \ar[r]\ar[ru]\ar[rd]& P_1/P_3\ar[r]&N\ar[r]\ar[ru]\ar[rd]&I_3\ar[r]&S_1\ar[ru]\ar[r]\ar[rd]& S_3[1]\ar[r]&P_1[1]\ar[ru]\ar[r]\ar[rd]&P_1/P_3[1]\ar[r]&N[1]\ar[ru]\ar[r]\ar[rd]&I_3[1]\ar[r]&S_4[1]&\\
   S_4\ar[ru] && P_1/P_4\ar[ru] && I_4\ar[ru] &&S_2[1]\ar[ru]&&P_1/P_2[1]\ar[ru]&&I_2[1]\ar[ru]}
\end{align*}
The algebra $(\Gamma^1)^{op}$ is given by the opposite quiver below with mesh relations.

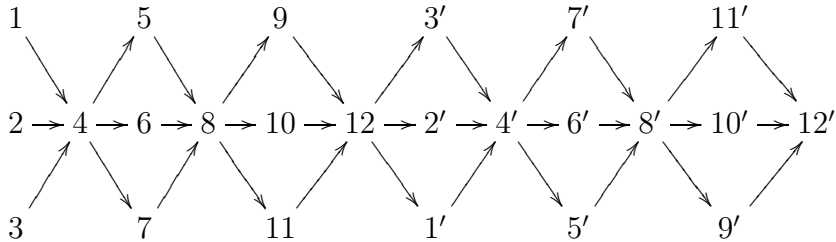
\begin{align*}
\xymatrixcolsep{4mm}
\xymatrix{
   1\ar[rd] && 5\ar[rd] && 9\ar[rd] && 3'\ar[rd] && 7'\ar[rd]&&11'\ar[rd]&\\ 
   2\ar[r]& 4 \ar[r]\ar[ru]\ar[rd]& 6\ar[r]&8\ar[r]\ar[ru]\ar[rd]&10\ar[r]&12\ar[ru]\ar[r]\ar[rd]& 2'\ar[r]&4'\ar[ru]\ar[r]\ar[rd]&6'\ar[r]&8'\ar[ru]\ar[r]\ar[rd]&10'\ar[r]&12'&\\
   3\ar[ru] && 7\ar[ru] && 11\ar[ru] &&1'\ar[ru]&&5'\ar[ru]&&9'\ar[ru]}
\end{align*}

$\Sigma^1$ is given by the quiver below with mesh relations.

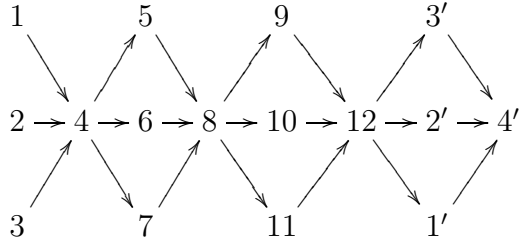
\begin{align*}
\xymatrixcolsep{4mm}
\xymatrix{
   1\ar[rd] && 5\ar[rd] && 9\ar[rd] && 3'\ar[rd] && &&&\\ 
   2\ar[r]& 4 \ar[r]\ar[ru]\ar[rd]& 6\ar[r]&8\ar[r]\ar[ru]\ar[rd]&10\ar[r]&12\ar[ru]\ar[r]\ar[rd]& 2'\ar[r]&4'&&&&&\\
   3\ar[ru] && 7\ar[ru] && 11\ar[ru] &&1'\ar[ru]&&&&}
\end{align*}

\subsection{Example $\mathbb{A}^2_3$}
Let $\cM$ be the 2-cluster tilting object of Auslander algebra of oriented $A_3$ quiver which is

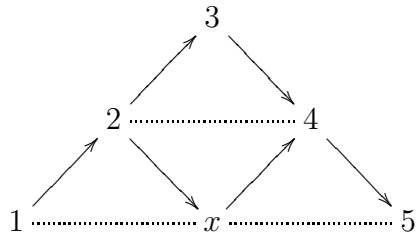
\begin{align*}
\xymatrix{
&&3\ar[rd]&&\\
&2\ar[ru]\ar@{..}[rr]\ar[rd]&&4\ar[rd]&\\
1\ar[ru]\ar@{..}[rr]&&x\ar@{..}[rr]\ar[ru]&&5}
\end{align*}
with mesh relations.

 The subcategory $\cM\subset\modd \End_{A}(G)$ where $A=\mathbb{K}A_3$ and $\add G=\modd\mathbb{K}A_3$ is 

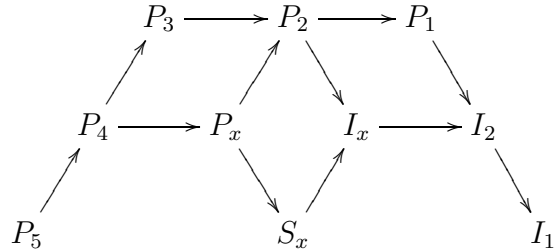
\begin{align*}
\xymatrixcolsep{2mm}
\xymatrix{
&&P_3\ar[rr]&&P_2\ar[rr]\ar[rd]&&P_1\ar[rd]&&&\\
&P_4\ar[ru]\ar[rr]&&P_x\ar[ru]\ar[rd]&&I_x\ar[rr]&&I_2\ar[rd]&&\\
P_5\ar[ru]&&&&S_x\ar[ru]&&&&I_1&
}
\end{align*}
The Auslander-Reiten quiver of $\cM^2$ is
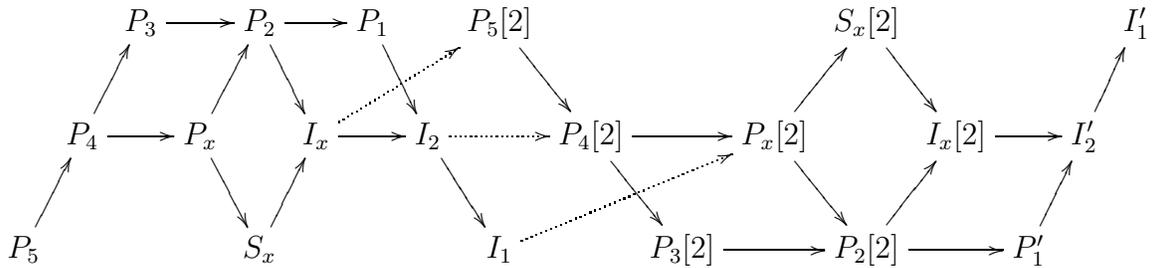
\[
\xymatrixcolsep{1mm}
\xymatrix{
&&P_3\ar[rr]&&P_2\ar[rr]\ar[rd]&&P_1\ar[rd]&&P_5[2]\ar[rd]&&&&S_x[2]\ar[rd]&&&& I'_1\\
&P_4\ar[ru]\ar[rr]&&P_x\ar[ru]\ar[rd]&&I_x\ar@{..>}[rrru]\ar[rr]&&I_2\ar@{..>}[rr]\ar[rd]&&P_4[2]\ar[rr]\ar[rd]&&P_x[2]\ar[ru]\ar[rd]&&I_x[2]\ar[rr]&&I'_2\ar[ru]\\
P_5\ar[ru]&&&&S_x\ar[ru]&&&&I_1\ar@{..>}[rrru]&&P_3[2]\ar[rr]&&P_2[2]\ar[ru]\ar[rr]&&P'_1\ar[ru]
}
\]

Connecting 4-angles are

\begin{align*}
P_5\rightarrow P_3\rightarrow  P_2\rightarrow  I_x\rightarrow  P_5[2]\\
P_4\rightarrow  P_3\rightarrow  P_1\rightarrow  I_2\rightarrow  P_4[2]\\
P_x\rightarrow  P_2\rightarrow  P_1\rightarrow  I_1\rightarrow  P_x[2]\\
S_x\rightarrow   I_x \rightarrow I_2\rightarrow  I_1\rightarrow  S_x[2]
\end{align*}

The algebra $\Sigma^1$ is equivalent to the endomorphism algebra of $\cM\oplus A[2]$ which is

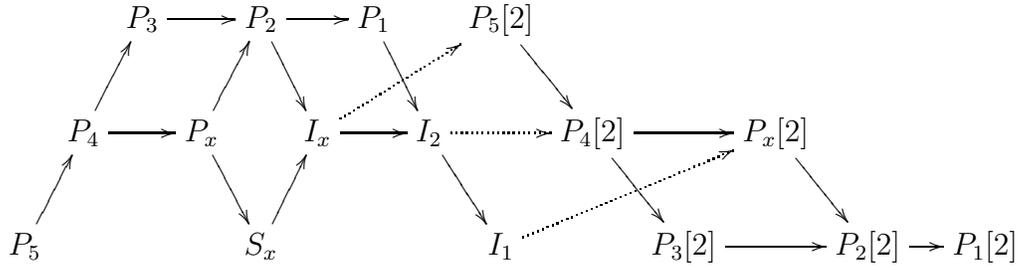
\begin{align*}
\xymatrixcolsep{1mm}
\xymatrix{
&&P_3\ar[rr]&&P_2\ar[rr]\ar[rd]&&P_1\ar[rd]&&P_5[2]\ar[rd]&&&&&&&& \\
&P_4\ar[ru]\ar[rr]&&P_x\ar[ru]\ar[rd]&&I_x\ar@{..>}[rrru]\ar[rr]&&I_2\ar@{..>}[rr]\ar[rd]&&P_4[2]\ar[rr]\ar[rd]&&P_x[2]\ar[rd]&&&&\\
P_5\ar[ru]&&&&S_x\ar[ru]&&&&I_1\ar@{..>}[rrru]&&P_3[2]\ar[rr]&&P_2[2]\ar[rr]&&P_1[2]
}
\end{align*}
Notice that it is $3$-dimensional Nakayama algebra given by Kupisch series $(1,2,3,3)$.

\subsection{Example $G_2$}
Let $B$ be an algebra of rank $6$ whose projective objects are given by\\
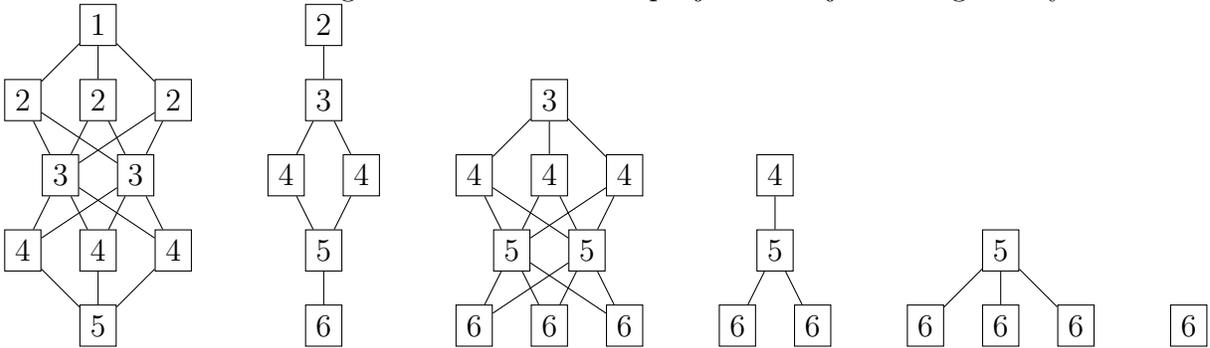
\begin{minipage}{.5\textwidth}
\centering

\begin{tikzpicture}
 
 \node at (3,0) [rectangle,draw] (1) {$1$};
 \node at (2,-1) [rectangle,draw] (2) {$2$};
 \node at (3,-1) [rectangle,draw] (3) {$2$};
 \node at (4,-1) [rectangle,draw] (4) {$2$};
 \node at (2.5,-2) [rectangle,draw] (5) {$3$};
 \node at (3.5,-2) [rectangle,draw] (6) {$3$};
 \node at (2,-3) [rectangle,draw] (7) {$4$};
 \node at (3,-3) [rectangle,draw] (8) {$4$};
 \node at (4,-3) [rectangle,draw] (9) {$4$};
 \node at (3,-4) [rectangle,draw] (10) {$5$};

 \draw (1)--(2);
 \draw (1)--(3);
  \draw (1)--(4);
 \draw (2)--(5);  
  \draw (2)--(6);
   \draw (3)--(5);
    \draw (3)--(6);
    
     \draw (4)--(5);
      \draw (4)--(6);
       \draw (5)--(7);
        \draw (5)--(8);
         \draw (5)--(9);
          \draw (6)--(7);
           \draw (6)--(8);
            \draw (6)--(9);
             \draw (7)--(10);
              \draw (8)--(10);
               \draw (9)--(10);
\begin{scope}[xshift={3cm}]
 \node at (3,0) [rectangle,draw] (1) {$2$};
  \node at (3,-1) [rectangle,draw] (3) {$3$};
  \node at (2.5,-2) [rectangle,draw] (5) {$4$};
 \node at (3.5,-2) [rectangle,draw] (6) {$4$};
  \node at (3,-3) [rectangle,draw] (8) {$5$};
 \node at (3,-4) [rectangle,draw] (9) {$6$};
 \draw (1)--(3);\draw (3)--(5);\draw (3)--(6);\draw (5)--(8);\draw (6)--(8);\draw (8)--(9);
   \end{scope}

  \begin{scope}[xshift={6cm}]

 \node at (3,-1) [rectangle,draw] (6) {$3$};
 \node at (2,-2) [rectangle,draw] (7) {$4$};
 \node at (3,-2) [rectangle,draw] (8) {$4$};
 \node at (4,-2) [rectangle,draw] (9) {$4$};
 \node at (3.5,-3) [rectangle,draw] (10) {$5$};
  \node at (2.5,-3) [rectangle,draw] (11) {$5$};
  \node at (3,-4) [rectangle,draw] (12) {$6$};
   \node at (2,-4) [rectangle,draw] (13) {$6$};
    \node at (4,-4) [rectangle,draw] (14) {$6$};
    
\draw (6)--(7);\draw (7)--(11);\draw (11)--(12);
  \draw (6)--(8);\draw (8)--(11);\draw (11)--(13);
  \draw (6)--(9);\draw (9)--(11);\draw (11)--(14);
    \draw (7)--(10);\draw (10)--(12);
      \draw (8)--(10);\draw (10)--(13);
         \draw (9)--(10);\draw (10)--(14);
  \end{scope}
  
   \begin{scope}[xshift={9cm}]

 \node at (3,-2) [rectangle,draw] (7) {$4$};
   \node at (3,-3) [rectangle,draw] (11) {$5$};
  \node at (2.5,-4) [rectangle,draw] (12) {$6$};
   \node at (3.5,-4) [rectangle,draw] (13) {$6$};
          \draw (7)--(11);\draw (11)--(12);
           \draw (11)--(13);
  \end{scope}
  \begin{scope}[xshift={14cm}]

   \node at (3.5,-4) [rectangle,draw] (13) {$6$};
       
  \end{scope}
  \begin{scope}[xshift={12cm}]

   \node at (3,-3) [rectangle,draw] (11) {$5$};
  \node at (2,-4) [rectangle,draw] (12) {$6$};
   \node at (3,-4) [rectangle,draw] (13) {$6$};
   \node at (4,-4) [rectangle,draw] (14) {$6$};
          \draw (11)--(14);\draw (11)--(12);
           \draw (11)--(13);
  \end{scope}
\end{tikzpicture}

\end{minipage}
\vspace{1cm}\\
It follows that $\gldim B=2=\domdim B$, so it is a Auslander algebra. Now we consider the following algebra $B'$ whose projective objects are

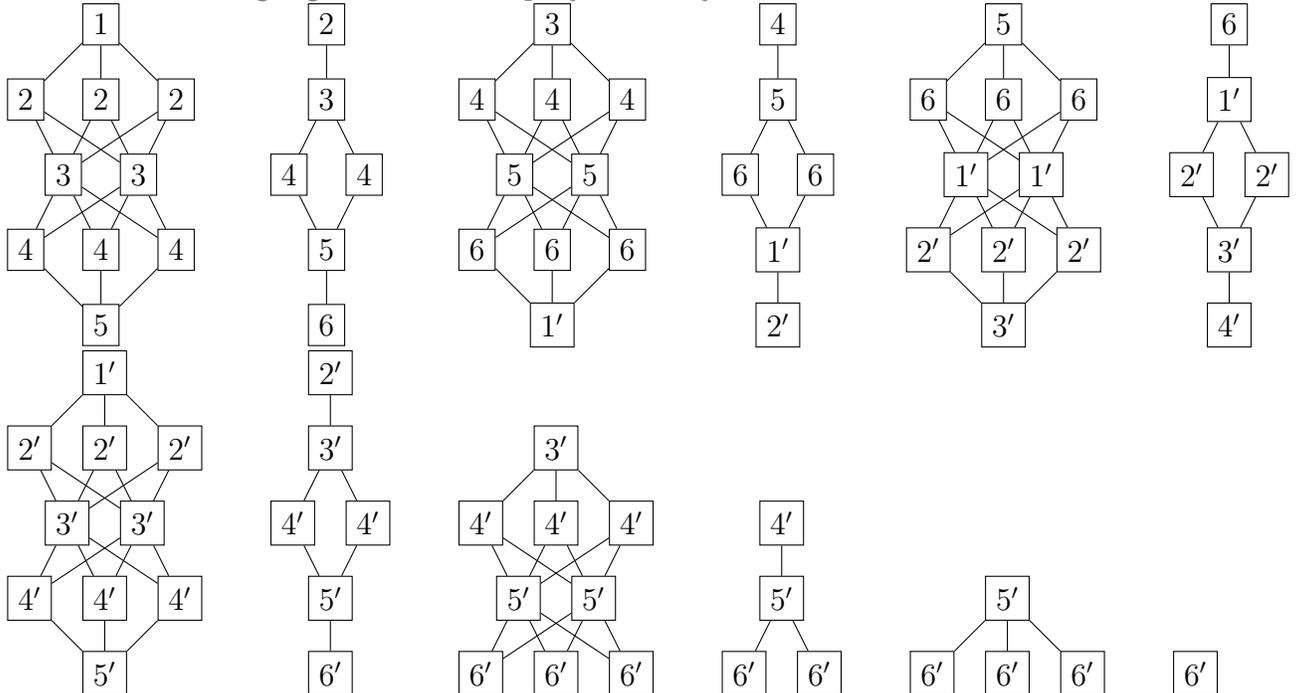
\begin{minipage}{.5\textwidth}
\centering

\begin{tikzpicture}
 
 \node at (3,0) [rectangle,draw] (1) {$1$};
 \node at (2,-1) [rectangle,draw] (2) {$2$};
 \node at (3,-1) [rectangle,draw] (3) {$2$};
 \node at (4,-1) [rectangle,draw] (4) {$2$};
 \node at (2.5,-2) [rectangle,draw] (5) {$3$};
 \node at (3.5,-2) [rectangle,draw] (6) {$3$};
 \node at (2,-3) [rectangle,draw] (7) {$4$};
 \node at (3,-3) [rectangle,draw] (8) {$4$};
 \node at (4,-3) [rectangle,draw] (9) {$4$};
 \node at (3,-4) [rectangle,draw] (10) {$5$};

 \draw (1)--(2); \draw (1)--(3);  \draw (1)--(4);
 \draw (2)--(5);    \draw (2)--(6);   \draw (3)--(5);
    \draw (3)--(6);\draw (4)--(5);\draw (4)--(6);
\draw (5)--(7); \draw (5)--(8);  \draw (5)--(9);
 \draw (6)--(7);\draw (6)--(8);\draw (6)--(9);
 \draw (7)--(10); \draw (8)--(10); \draw (9)--(10);
 
\begin{scope}[xshift={3cm}]
 \node at (3,0) [rectangle,draw] (1) {$2$};
  \node at (3,-1) [rectangle,draw] (3) {$3$};
  \node at (2.5,-2) [rectangle,draw] (5) {$4$};
 \node at (3.5,-2) [rectangle,draw] (6) {$4$};
  \node at (3,-3) [rectangle,draw] (8) {$5$};
 \node at (3,-4) [rectangle,draw] (9) {$6$};
 \draw (1)--(3);\draw (3)--(5);\draw (3)--(6);\draw (5)--(8);\draw (6)--(8);\draw (8)--(9);
   \end{scope}

  \begin{scope}[xshift={6cm}]

 \node at (3,0) [rectangle,draw] (6) {$3$};
 \node at (2,-1) [rectangle,draw] (7) {$4$};
 \node at (3,-1) [rectangle,draw] (8) {$4$};
 \node at (4,-1) [rectangle,draw] (9) {$4$};
 \node at (3.5,-2) [rectangle,draw] (10) {$5$};
  \node at (2.5,-2) [rectangle,draw] (11) {$5$};
  \node at (3,-3) [rectangle,draw] (12) {$6$};
   \node at (2,-3) [rectangle,draw] (13) {$6$};
    \node at (4,-3) [rectangle,draw] (14) {$6$};
    \node at (3,-4) [rectangle,draw] (15) {$1'$};
    
\draw (6)--(7);\draw (7)--(11);\draw (11)--(12);
  \draw (6)--(8);\draw (8)--(11);\draw (11)--(13);
  \draw (6)--(9);\draw (9)--(11);\draw (11)--(14);
    \draw (7)--(10);\draw (10)--(12);\draw (12)--(15);
      \draw (8)--(10);\draw (10)--(13);\draw (13)--(15);
         \draw (9)--(10);\draw (10)--(14);\draw (14)--(15);
  \end{scope}
  
   \begin{scope}[xshift={9cm}]

 \node at (3,0) [rectangle,draw] (7) {$4$};
   \node at (3,-1) [rectangle,draw] (11) {$5$};
  \node at (2.5,-2) [rectangle,draw] (12) {$6$};
   \node at (3.5,-2) [rectangle,draw] (13) {$6$};
    \node at (3,-3) [rectangle,draw] (1) {$1'$};
   \node at (3,-4) [rectangle,draw] (2) {$2'$};
          \draw (7)--(11);\draw (11)--(12);
           \draw (11)--(13);\draw (12)--(1);
           \draw (13)--(1);\draw (1)--(2);
  \end{scope}
  \begin{scope}[xshift={15cm}]

  \node at (3,0) [rectangle,draw] (1) {$6$};
  \node at (3,-1) [rectangle,draw] (3) {$1'$};
  \node at (2.5,-2) [rectangle,draw] (5) {$2'$};
 \node at (3.5,-2) [rectangle,draw] (6) {$2'$};
  \node at (3,-3) [rectangle,draw] (8) {$3'$};
 \node at (3,-4) [rectangle,draw] (9) {$4'$};
 \draw (1)--(3);\draw (3)--(5);\draw (3)--(6);\draw (5)--(8);\draw (6)--(8);\draw (8)--(9);
       
  \end{scope}
  \begin{scope}[xshift={12cm}]

 \node at (3,0) [rectangle,draw] (1) {$5$};
 \node at (2,-1) [rectangle,draw] (2) {$6$};
 \node at (3,-1) [rectangle,draw] (3) {$6$};
 \node at (4,-1) [rectangle,draw] (4) {$6$};
 \node at (2.5,-2) [rectangle,draw] (5) {$1'$};
 \node at (3.5,-2) [rectangle,draw] (6) {$1'$};
 \node at (2,-3) [rectangle,draw] (7) {$2'$};
 \node at (3,-3) [rectangle,draw] (8) {$2'$};
 \node at (4,-3) [rectangle,draw] (9) {$2'$};
 \node at (3,-4) [rectangle,draw] (10) {$3'$};

 \draw (1)--(2); \draw (1)--(3);  \draw (1)--(4);
 \draw (2)--(5);    \draw (2)--(6);   \draw (3)--(5);
    \draw (3)--(6);\draw (4)--(5);\draw (4)--(6);
\draw (5)--(7); \draw (5)--(8);  \draw (5)--(9);
 \draw (6)--(7);\draw (6)--(8);\draw (6)--(9);
 \draw (7)--(10); \draw (8)--(10); \draw (9)--(10);
  \end{scope}
\end{tikzpicture}

\end{minipage}

\begin{minipage}{.5\textwidth}
\centering

\begin{tikzpicture}
 
 \node at (3,0) [rectangle,draw] (1) {$1'$};
 \node at (2,-1) [rectangle,draw] (2) {$2'$};
 \node at (3,-1) [rectangle,draw] (3) {$2'$};
 \node at (4,-1) [rectangle,draw] (4) {$2'$};
 \node at (2.5,-2) [rectangle,draw] (5) {$3'$};
 \node at (3.5,-2) [rectangle,draw] (6) {$3'$};
 \node at (2,-3) [rectangle,draw] (7) {$4'$};
 \node at (3,-3) [rectangle,draw] (8) {$4'$};
 \node at (4,-3) [rectangle,draw] (9) {$4'$};
 \node at (3,-4) [rectangle,draw] (10) {$5'$};

 \draw (1)--(2);
 \draw (1)--(3);
  \draw (1)--(4);
 \draw (2)--(5);  
  \draw (2)--(6);
   \draw (3)--(5);
    \draw (3)--(6);
    
     \draw (4)--(5);
      \draw (4)--(6);
       \draw (5)--(7);
        \draw (5)--(8);
         \draw (5)--(9);
          \draw (6)--(7);
           \draw (6)--(8);
            \draw (6)--(9);
             \draw (7)--(10);
              \draw (8)--(10);
               \draw (9)--(10);
\begin{scope}[xshift={3cm}]
 \node at (3,0) [rectangle,draw] (1) {$2'$};
  \node at (3,-1) [rectangle,draw] (3) {$3'$};
  \node at (2.5,-2) [rectangle,draw] (5) {$4'$};
 \node at (3.5,-2) [rectangle,draw] (6) {$4'$};
  \node at (3,-3) [rectangle,draw] (8) {$5'$};
 \node at (3,-4) [rectangle,draw] (9) {$6'$};
 \draw (1)--(3);\draw (3)--(5);\draw (3)--(6);\draw (5)--(8);\draw (6)--(8);\draw (8)--(9);
   \end{scope}

  \begin{scope}[xshift={6cm}]

 \node at (3,-1) [rectangle,draw] (6) {$3'$};
 \node at (2,-2) [rectangle,draw] (7) {$4'$};
 \node at (3,-2) [rectangle,draw] (8) {$4'$};
 \node at (4,-2) [rectangle,draw] (9) {$4'$};
 \node at (3.5,-3) [rectangle,draw] (10) {$5'$};
  \node at (2.5,-3) [rectangle,draw] (11) {$5'$};
  \node at (3,-4) [rectangle,draw] (12) {$6'$};
   \node at (2,-4) [rectangle,draw] (13) {$6'$};
    \node at (4,-4) [rectangle,draw] (14) {$6'$};
    
\draw (6)--(7);\draw (7)--(11);\draw (11)--(12);
  \draw (6)--(8);\draw (8)--(11);\draw (11)--(13);
  \draw (6)--(9);\draw (9)--(11);\draw (11)--(14);
    \draw (7)--(10);\draw (10)--(12);
      \draw (8)--(10);\draw (10)--(13);
         \draw (9)--(10);\draw (10)--(14);
  \end{scope}
  
   \begin{scope}[xshift={9cm}]

 \node at (3,-2) [rectangle,draw] (7) {$4'$};
   \node at (3,-3) [rectangle,draw] (11) {$5'$};
  \node at (2.5,-4) [rectangle,draw] (12) {$6'$};
   \node at (3.5,-4) [rectangle,draw] (13) {$6'$};
          \draw (7)--(11);\draw (11)--(12);
           \draw (11)--(13);
  \end{scope}
  \begin{scope}[xshift={14cm}]

   \node at (3.5,-4) [rectangle,draw] (13) {$6'$};
       
  \end{scope}
  \begin{scope}[xshift={12cm}]

   \node at (3,-3) [rectangle,draw] (11) {$5'$};
  \node at (2,-4) [rectangle,draw] (12) {$6'$};
   \node at (3,-4) [rectangle,draw] (13) {$6'$};
   \node at (4,-4) [rectangle,draw] (14) {$6'$};
          \draw (11)--(14);\draw (11)--(12);
           \draw (11)--(13);
  \end{scope}
\end{tikzpicture}

\end{minipage}

One can verify that $\gldim B'=5=\domdim B'$. Moreover, by choice of simple modules, it is duplicated algebra of $B$. We constructed this algebra by looking not Auslander-Reiten quiver of Dynkin quiver $G_2$ but from the following quiver which presents all irreducible maps of $\modd G_2$.

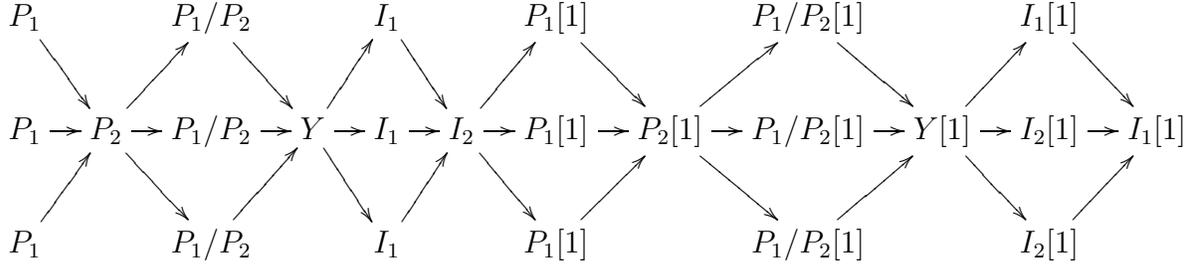
\begin{align*}
\xymatrixcolsep{4mm}
\xymatrix{
   P_1\ar[rd] && P_1/P_2\ar[rd] && I_1\ar[rd] && P_1[1]\ar[rd] && P_1/P_2[1]\ar[rd]&&I_1[1]\ar[rd]&\\
   P_1\ar[r]& P_2 \ar[r]\ar[ru]\ar[rd]& P_1/P_2\ar[r]&Y\ar[r]\ar[ru]\ar[rd]&I_1\ar[r]&I_2\ar[ru]\ar[r]\ar[rd]& P_1[1]\ar[r]&P_2[1]\ar[ru]\ar[r]\ar[rd]&P_1/P_2[1]\ar[r]&Y[1]\ar[ru]\ar[r]\ar[rd]&I_2[1]\ar[r]&I_1[1]&\\
   P_1\ar[ru] && P_1/P_2\ar[ru] && I_1\ar[ru] &&P_1[1]\ar[ru]&&P_1/P_2[1]\ar[ru]&&I_2[1]\ar[ru]}
\end{align*}
where $P_1, P_2$ are the modules $0\rightarrow F$, $G\rightarrow F$ where $[G:F]$ is degree three extension of field $F$. Moreover, one can verify that similar constructions work for non simply laced Dynkin quivers. 
We wonder whether there exists higher analogues of K-species studied in \cite{dlab1976indecomposable} from the higher homological algebra point of view.

\bibliographystyle{alpha}
%\bibliography{derivedcat}

\bibliographystyle{alpha}

\end{document}